\documentclass[english,12pt]{amsart} \usepackage[utf8]{inputenc}
\usepackage[T1]{fontenc} \usepackage{ucs} \usepackage[top=1in,
bottom=1in, left=1in, right=1in]{geometry} \usepackage{color}

\usepackage{amsmath} \usepackage{amsfonts,dsfont}
\usepackage[psamsfonts]{amssymb}
\usepackage{mathtools} 
\usepackage{amsthm} \usepackage{mathrsfs} 
\usepackage{graphicx}
\usepackage{enumerate} 
\usepackage{hyperref,soul} \usepackage{soul} \usepackage{comment}
\usepackage{mathabx}

 \newtheorem{theoA}{Theorem}

\newtheorem{propA}[theoA]{Proposition}

\newtheorem{lemA}[theoA]{Lemma} 
\newtheorem{coroA}[theoA]{Corollary}

\renewcommand{\theequation}{\thesection.\arabic{equation}}
\newtheorem{thm}{Theorem} 
\newtheorem{lem}{Lemma}
\newtheorem{prop}{Proposition} 
\newtheorem{coro}{Corollary}
\newtheorem{rmk}{Remark}

\newtheorem{expl}{Example}

\newtheorem{notation}{Notations} 
\usepackage{thmtools}
\declaretheoremstyle[notefont=\bfseries,notebraces={}{},%
headpunct={},postheadspace=1em]{mystyle}
\declaretheorem[style=mystyle,numbered=no,name=Theorem]{thm-hand}
\newcommand{\eqnsection}{
  \renewcommand{\theequation}{\thesection.\arabic{equation}}
  \makeatletter \csname @addtoreset\endcsname{equation}{section}
  \makeatother} \eqnsection

\graphicspath{{./images/}
} 
\declaretheoremstyle[notefont=\bfseries,notebraces={}{},%
headpunct={},postheadspace=1em]{mystyle}
\DeclareMathOperator{\trace}{Trace}

\def\P{{\mathbb{P}}} \def\E{{\mathbb{E}}} \def\R{{\mathbb{R}}}
\def\N{{\mathbb{N}}}  \def\demi{{1\over 2}}
 \def\B{{\mathbb{B}}} \def\fff{{\mathcal F}}
\def\aaa{{\mathcal A}}

\begin{document}
\date{}
\title[Hitting times of interacting drifted BM and the VRJP]{Hitting
  times of interacting drifted Brownian motions and the Vertex
  Reinforced Jump Process} \author[C. Sabot]{Christophe SABOT}
\address{Universit\'e de Lyon, Universit\'e Lyon 1, Institut Camille
  Jordan, CNRS UMR 5208, 43, Boulevard du 11 novembre 1918, 69622
  Villeurbanne Cedex, France} \email{sabot@math.univ-lyon1.fr}
\author[X. Zeng]{Xiaolin Zeng} \thanks{This work was supported by the LABEX MILYON (ANR-10-LABX-0070) of Universit\'e de Lyon, within the program "Investissements d'Avenir" (ANR-11-IDEX-0007) operated by the French National Research Agency (ANR), and by the ANR/FNS project MALIN (ANR-16-CE93-0003). The second author is supported
  by ERC Starting Grant 678520.}  \address{School of
  mathematics, Tel aviv university, 108 Schreiber building,
  P.O.B. 39040, Ramat aviv, Tel aviv 69978, Israel}
\email{xiaolinzeng@mail.tau.ac.il}
\begin{abstract}
Consider a negatively drifted one dimensional Brownian motion starting
at positive initial position, its first hitting time to 0 has the
inverse Gaussian law. Moreover, conditionally on this hitting time,
the Brownian motion up to that time has the law of a 3-dimensional
Bessel bridge. In this paper, we give a generalization of this result
to a family of Brownian motions with interacting drifts, indexed by
the vertices of a conductance network. The hitting times are equal in
law to the inverse of a random potential that appears in the analysis
of a self-interacting process called the Vertex Reinforced Jump
Process (\cite{STZ15,SZ15}). These Brownian motions with interacting drifts
have remarkable properties with respect to restriction and
conditioning, showing hidden Markov properties. This family of
processes are closely related to the martingale
that plays a crucial role in the analysis of the vertex reinforced
jump process and edge reinforced random walk (\cite{SZ15}) on infinite
graphs.
\end{abstract}
\maketitle
\section{Introduction}
We first recall some classic facts about hitting times of standard
Brownian motion.  Let \((B_t)_{t\ge 0}\) be a standard Brownian motion
and
\[
X(t)= \theta+ B(t),
\]
be a Brownian motion starting from initial position \(\theta>0\). It
is well-known that the first hitting time of 0
\begin{equation}\label{Brownian}
T=\inf\{t\ge 0, \; X(t)=0\}
\end{equation}
has the law of the inverse of a Gamma random variable with parameter
\((\demi, {\theta^2\over 2})\).  Moreover, conditionally on \(T\),
\((X_{t})_{0\le t\le T}\) has the law of a 3-dimensional Bessel bridge
from \(\theta\) to \(0\) on time interval \([0,T]\)\footnote{The 3-dimensional Bessel bridge from \(\theta\) to 0 on
  time interval \([0,T]\) can be represented by the following S.D.E. 
  \[
  X(t)= \theta+B(t)+\int_0^t\left({1\over X(s)}-{X(s)\over T
      -s}\right)ds, \ 0\le t\le T
  \]
}.  More generally, if
\begin{equation}\label{drifted}
X(t)= \theta+ B(t)-\eta t,
\end{equation}
is a drifted Brownian motion with negative drift \(-\eta<0\) starting
at \(\theta>0\), then \(T\) has the inverse Gaussian distribution
with parameters \(({\theta\over \eta}, {\theta^2})\), i.e. \(T\) has
density
\[
f(t)=\frac{\theta}{\sqrt{2\pi t^{3}}}\exp\left(-\demi\left({\theta^2\over t}+{\eta^2 t}-2\eta\theta\right)
\right) \mathds{1}_{t>0} dt.
\]
Moreover, conditionally on \(T\), \((X_{t})_{0\le t\le T}\) has the law of a
3-dimensional Bessel bridge from \(\theta\) to \(0\) on time interval
\([0,T]\). (See
\cite{MR0350881}, Theorem 3.1, or \cite{Revuz-Yor}, p.~317 Corollary~4.6, and \cite{MR0375485,MR1117781} for complements)

This paper aims at giving a generalization of these statements on a conductance network, namely for a family of Brownian motions with interacting drifts indexed by the vertices
of the network.  The distribution of hitting times of these
processes will be given by a multivariate exponential family of
distributions introduced by Sabot, Tarr\`es and Zeng~\cite{STZ15}, and generalized in \cite{Letac,letac2017multivariate}, which
appeared in the context of self-interacting processes and random
Schr\"odinger operators. This family of distributions is also intimately
related to the supersymmetric hyperbolic sigma model introduced by
Zirnbauer~\cite{Zirnbauer91} and investigated by Disertori, Spencer,
Zirnbauer~\cite{DSZ06,DS10}, and plays a crucial role in the
analysis of the edge reinforced random walk (ERRW) and the vertex
reinforced jump process (VRJP)~\cite{ST15,disertori2014transience,SZ15}.

The generalization of the one dimensional statement presented in this introduction
was hinted by the martingales that appear in \cite{SZ15}. This martingale has played
an important role in the analysis of the ERRW and the VRJP on infinite
graphs. In Section \ref{link_VRJP}, we
explain the relations between the stochastic differential equations
(S.D.E.s) defined in this paper and the VRJP and in Section~\ref{sec_martingale} we relate the martingales that appear in the study of VRJP to the S.D.E.s.

Note that the computations done in this paper seem to have many similarities with computations done for exponential functional of the Brownian motion in dimension one
(see in particular Matsumoto, Yor~\cite{MR2203675,MR2203676,MR1833628}). More precisely, it would be possible to write an analogue of the Lamperti transformation that changes
the S.D.E. \eqref{SDE0} presented below in its exponential functional counterpart with $\mu=\demi$ (see the Matsumoto Yor opposite drift theorem \cite{MR1833628}): the counterpart of the representation of Theorem~\ref{thm_main1} would correspond to a representation of the S.D.E. with 
a Brownian motions with opposite drifts as in \cite{MR1833628}. In fact, in dimension one (i.e. one vertex), the Inverse Gaussian distribution corresponds to \(\mu=\frac{1}{2}\), and the Generalized Inverse Gaussian (GIG) distribution corresponds to general \(\mu\in \mathbb{R}\), see \cite{barndorff1978hyperbolic} and \cite{MR1117781}. On a conductance network (i.e. multidimensional), the case \(\mu=\frac{1}{2}\) can be carried out by explicit computation, for general \(\mu\), one will have to use Bessel K functions as normalizing constant. We plan to develop these aspects in a further work.

It might not be a coincidence that the GIG distribution was initially called generalized hyperbolic distribution, and the distribution we considered here stems from a supersymmetric hyperbolic sigma model, where one considered spin systems with spins taking values on a super hyperbolic space. Interested readers can check \cite{barndorff1978hyperbolic} and \cite{spencer2012susy} for more details.

Another related direction goes back to Vallois, where GIG is conceived as the exit law of some one dimensional diffusion. In \cite{MR3418540}, Chhaibi explicitly computed the exit law of certain hypoelliptic Brownian motion on a solvable Lie group, where e.g. he recovered the Matsumoto Yor opposite drift theorem, by taking the group to be $\mathfrak{sl}_{2}$. It is very likely that there is a connection with our work. Note also that the integral of a geometric Brownian motion is closely related to the study of Asian option. At last, some related open questions are listed in Section 4.5 of \cite{letac2017multivariate}.

\section{Statement of the main results}
\subsection{The multivariate generalization of inverse Gaussian law :
  the random potential associated with the VRJP}
Let \(N\) be a positive integer and \(V=\{1, \ldots, N\}\). Given a symmetric matrix
\[
W=(W_{i,j})_{i,j=1, \ldots, N}
\]
with non negative coefficients \(W_{i,j}=W_{j,i}\ge 0\).  We denote by
\(\mathcal{G}=(V,E)\) the associated graph with:
\[V=\{1 ,\ldots,N\} \text{ and }E=\{\{i,j\},\; i\neq j, \; W_{i,j}>0\}.
\]
We always assume that the matrix \(W\) is irreducible, i.e. the
graph \(\mathcal{G}\) is connected.  If \((\beta_i)_{i\in V}\) is a vector indexed by the vertices, we set
\begin{equation}\label{Hbeta}
H_\beta= 2\beta -W,
\end{equation}
where \(2\beta\) represents the operator of multiplication by the
vector \((2\beta_i)\) (or equivalently the diagonal matrix with
diagonal coefficients \((2\beta_i)_{i\in V}\)).  We always write
\(H_\beta>0\) to mean that \(H_\beta\) is positive definite.  Remark
that when \(H_\beta>0\), all the entries of
\((H_\beta)^{-1}\) are positive (since \(\mathcal{G}\) is connected and \(H_{\beta}\) is an M-matrix, see e.g.  \cite{berman1994nonnegative}, Proposition 3).

The following distribution was introduced in \cite{STZ15}, and
generalized in \cite{Letac,letac2017multivariate}.
\begin{lemA}\label{lem_beta}
Let \((\theta_i)_{i\in V}\in (\R_+^*)^V\) be a positive vector indexed
by \(V\). Let \((\eta_i)_{i\in V}\in (\R_+)^V\) be a non negative
vector indexed by \(V\).  The measure
\begin{equation} \label{eq:1} \nu_V^{W,\theta,\eta}(d\beta):=
\mathds{1}_{H_{\beta}>0}\left(\frac{2}{\pi}\right)^{|V|/2}\exp\left(-\demi\left<\theta,
    H_\beta \theta\right> -\demi\left<\eta, H_\beta^{-1}
    \eta\right>+\left<\eta,\theta\right>\right)\frac{\prod_{i\in
    V}{{\theta_i}}}{\sqrt{\det H_{\beta}}} d\beta
\end{equation}
is a probability distribution on \(\R^V\), where
\( \mathds{1}_{H_{\beta}>0}\) is the indicator function that the
operator \(H_\beta\) (defined in \eqref{Hbeta}) is positive definite, \(\left< \cdot,\cdot \right>\) is the usual inner product on \(\mathbb{R}^{V}\), and
\(d\beta=\prod_{i\in V} d\beta_i\).  When \(\eta=0\), we simply write
\(\nu^{W,\theta}_V\) for \(\nu^{W,\theta,0}_V\).

Moreover, the Laplace transform of (\ref{eq:1}) is explicitly given by
\begin{equation}
\label{eq:2}
\int e^{-\left<\lambda, \beta\right>} \nu^{W,\theta,\eta}_V(d\beta) =
e^{-\demi\left<\sqrt{\theta^2+\lambda},W\sqrt{\theta^2+\lambda}\right>+\demi\left<\theta,W\theta\right>
  +\left<\eta, \theta-\sqrt{\theta^2+\lambda}\right>}\prod_{i\in V}\frac{\theta_i}{\sqrt{\theta_i^2+\lambda_{i}}}
\end{equation}
for all $(\lambda_i)_{i\in V}$ such that \(\lambda_{i}+\theta_{i}^{2}> 0\),  \(\forall i\in V\).
\end{lemA}
\begin{rmk}
The probability distribution \(\nu_V^{W,\theta,\eta}\) was initially defined in
\cite{STZ15} in the case \(\eta=0\).
In \cite{Letac, letac2017multivariate}, Letac gave a shorter proof of the fact that \(\nu^{W,\theta}_V\) is a
probability and remarked that the family can be generalized to the
family \(\nu_V^{W,\theta,\eta}\) above.  It appears, see forthcoming
Lemma~\ref{lem_marg_cond}, that the general family \(\nu_V^{W,\theta,\eta}\)
can be obtained from the family  \(\nu_V^{W,\theta}\) by taking marginal laws.
\end{rmk}
\begin{rmk} The definition of \(\nu_V^{W,\theta}\) is not strictly the
same as \(\nu_{V}^{W,\theta}\) in \cite{STZ15}.  Firstly, compared with the definition of
\cite{STZ15}, the parameter \(\theta_i\) above corresponds to
\(\sqrt{\theta_i}\) in \cite{STZ15}. It is in fact simpler to write
the formula as in \eqref{eq:2} since the quadratic form
\(\left<\theta, H_\beta \theta\right>\) appears naturally in the
density and since \(\theta_i\) will play the role of the initial value
in the forthcoming S.D.E.  Secondly, we do not assume here that the
diagonal coefficients of \(W\) are zero. It is obvious that the two
definitions are equivalent up to a translation of \(\beta_i\) by
\(W_{i,i}\). It will be more convenient here to allow this generality.
\end{rmk}
\begin{notation}\label{not_1}
To simplify notations, in the sequel, for any function
\(\zeta:V\mapsto \R\) and any subset \(U\subset V\), we write
\(\zeta_U\) for the restriction of \(\zeta\) to the subset \(U\). We
write \(d\beta_{U}=\prod_{i\in U}d\beta_{i}\) to denote integration on
variables in \(\beta_U\). Similarly, if \(A\) is a \(V\times V\) matrix
and \(U\subset V\), \(U'\subset V\), we write \(A_{U,U'}\) for its
restriction to the block \(U\times U'\). Note also that when $(\xi_i)_{i\in V}$ is in $\R^V$, we sometimes simply write $\xi$ for the operator of
multiplication by $\xi$, (i.e. the diagonal matrix with diagonal coefficients \((\xi_i)_{i\in V}\)), 
as it is done in formula \eqref{Hbeta}. It will be clear from the context and considerations of dimension if it denotes a vector or the operator of multiplication.
Finally, we write
$\nu_U^{W,\theta,\eta}$ for $\nu_U^{W_{U,U},\theta_U,\eta_U}$ when
$U\subset V$ is a subset of $V$ and $W$ (resp. $\theta$, $\eta$) is a
$V\times V$ matrix (resp. vectors in $\R^V$).
\end{notation}

We state the counterpart of Proposition~1 of \cite{STZ15} in the
context of the measure \(\nu_V^{W,\theta, \eta}\).
\begin{coroA}\label{marginals-beta}
Under the probability distribution $\nu_V^W(d\beta)$, 
\begin{enumerate}[(i)]
\item the random variable \( {1\over 2\beta_i-W_{i,i}}\) follows an
inverse Gaussian distribution with parameters
\(({\theta_i\over \eta_i+\sum_{j\neq i} W_{i,j} \theta_j},
\theta_i^2)\), for all \(i\in V\),
\item the random vector \((\beta_i)\) is 1-dependent, i.e. for any
subsets \(V_1\subset V\), and \(V_2\subset V\) such that the distance
in the graph \(\mathcal{G}\) between \(V_1\) and \(V_2\) is strictly larger
than 1, then the random variables \(\beta_{V_1}\) and \(\beta_{V_2}\) are independent.
\end{enumerate}
\end{coroA}
The following lemma was proved independently in the 3rd arxiv version of
\cite{STZ15} and in \cite{letac2017multivariate}. (The result is stated in the case of \(\theta=1\) in \cite{STZ15},  Lemma 4, but it can
be easily extended to the case of general \(\theta\), see Section \ref{sec_results_distrib}).
\begin{lemA}\label{lem_marg_cond}
Let \(U\subset V\). Under the probability distribution \(\nu_V^{W,\theta,\eta}(d\beta)\),
\begin{enumerate}[(i)]
\item
\label{item-2}
 \(\beta_{U}\) is distributed according to
$\nu_U^{W,\theta,\eta}$ (i.e. \( \nu_{U}^{W_{U,U},\theta_U,\widehat\eta}\), c.f. Notations~\ref{not_1}) where
\begin{align}\label{eta}
  \widehat\eta= \eta_{U}+ W_{U,U^c} (\theta_{U^c}).
\end{align}
\item
\label{lem_conditioning}
conditionally on \(\beta_{U}\), \(\beta_{U^c}\) is distributed according to
\(\nu_{U^c}^{\widecheck W,\theta,\widecheck\eta}\) where
\[
\widecheck W=W_{U^c,U^c}+ W_{U^c,U} \left((H_\beta)_{U,U}\right)^{-1}
W_{U,U^c}, \;\;\; \widecheck \eta=\eta_{U^c}+W_{U^c,U}
\left((H_\beta)_{U,U}\right)^{-1}(\eta_{U}).
\]
\end{enumerate}
\end{lemA}

\subsection{Brownian motions with interacting drifts: main results}\label{subsec_results}
Let \(t^0=(t^0_{i})_{i\in V} \in (\R_+)^V\) be a nonnegative
vector. We set
\[
K_{t^0}=\operatorname{Id}-t^0W,
\]
where \(t^0\) denotes the operator of multiplication by \(t^0\) (or
equivalently the diagonal matrix with diagonal coefficients
\((t^0_i)\)).  Note that when \(t_{i}^{0}> 0\),  \(\forall i\in V\), we have
\(K_{t^0}=t^0(H_{{1\over 2t^0}})\), with notation \eqref{Hbeta} and \(\frac{1}{2t^{0}}=\left( \frac{1}{2t^{0}_{i}} \right)_{i\in V}\).

For \(T=(T_{i})_{i\in V}\in (\mathbb{R}_{+}\cup\{+\infty\})^{V}\) and
\(t\in \R_+\) we write \(t\wedge T\) for the vector
\((t\wedge T_i)_{i\in V}\), where for reals \(x,y\), \(x\wedge y=\min(x,y)\).

The following lemma introduces the processes which are the main objects of study of this paper as solution to a S.D.E.
\begin{lem}\label{def-SDE0}

Let \(\theta=(\theta_{i})_{i\in V}\in (\R_+)^V\) and
\(\eta=(\eta_{i})_{i\in V}\in (\R_+)^V\) be non-negative vectors. Denote \(|V|=N\), 
let \((B_i(t))_{i\in V}\) be a standard \(N\)-dimensional Brownian
motion.
\begin{enumerate}[(i)]
\item
\label{lem_i}
The following stochastic differential equation is well-defined for all
\(t\ge 0\) and has a unique pathwise solution :
\begin{gather*}\label{SDE0}\tag{\(E^{W,\theta,\eta}_{V}(Y)\)}
Y_{i}(t)= \theta_i+
\int_0^t \mathds{1}_{s<T_i}dB_i(s)-\int_0^t \mathds{1}_{s<T_i}
(W\psi(s))_ids,\;\;\; \forall i\in V,
\end{gather*}
where \(T=(T_{i})_{i\in V}\) is the random vector of stopping times
defined by
\begin{equation*}
T_{i}=\inf\{t\ge 0;\ Y_i(t)-t\eta_i=0\},\ \ \ \forall i\in V.
\end{equation*}
Also, \(\forall t\), \(K_{t\wedge T}\) is positive definite, and
\begin{align}
\label{psi0}
  \psi(t)=K_{t\wedge T}^{-1}\ Y(t)
\end{align}
Moreover, \(T_i<+\infty\) a.s. for all \(i\in V\), and \(K_T>0\) is positive definite.

\item
\label{lem_ii}
Denote \(X(t)=Y(t)-(t\wedge T)\eta\). The previous S.D.E is equivalent to
the following
\begin{gather*}\label{SDE0X}\tag{\(E^{W,\theta,\eta}_{V}(X)\)}
X_{i}(t)= \theta_i+ \int_0^t \mathds{1}_{s<T_i}dB_i(s)-\int_0^t
\mathds{1}_{s<T_i} ((W\psi)(s)+\eta)_ids,\;\;\; \forall i\in V,
\end{gather*}
with
\begin{align}
\label{psi0X}
  \psi(t)=K_{t\wedge T}^{-1}(X(t)+(t\wedge T)\eta)
\end{align}
and \(T_i\) is identified to the first hitting time of 0 by \(X_i(t)\).

\item
\label{lem_iii}
The process \(\psi(t)\) is a continuous vectorial martingale, it can be written as  (recall
that \(\mathds{1}_{s< T}\) is the operator of multiplication by
\(\mathds{1}_{s< T_{i}}\)) :
\begin{gather}\label{SDE2}\tag{\(E^{W,\theta,\eta}_{V}(\psi)\)}
\psi(t)= \theta+\int_0^t K_{s\wedge T}^{-1} \left(\mathds{1}_{s<T}
  dB(s)\right).
\end{gather}
Moreover,
the quadratic variation of \(\psi(t)\) is given by, for all \(t\ge 0\) (with convention that \(\frac{1}{\infty}=0,\frac{1}{0}=\infty\)),
\[
\left< \psi,\psi \right>_{t}= \left( H_{{1\over 2(t\wedge
      T)}}\right)^{-1}.
\]
\end{enumerate}
\end{lem}
It may not seem obvious at this point why we call these processes ``Brownian motions with interacting drifts''. The explanation will come at the end of this section as a consequence of the Abelian property Theorem~\ref{thm_abelian}: under the condition that the diagonal terms of $W$ are null, we will show that the marginals $(X_{i}(t))_{t\ge 0}$ are Brownian motions with constant negative drift stopped at their first hitting time of 0, see Corollary~\ref{Cor_marginals}.

Our first main result concerns the distribution of its hitting time:
\begin{thm}\label{thm_main1}
Let \((\theta_i)_{i\in V}\in (\R_+^*)^V\),
\((\eta_i)_{i\in V}\in (\R_+)^V\) and \((Y_i(t))_{i\in V}\), \((X_i(t))_{i\in V}\),
\((T_i)_{i\in V}\) be as in Lemma \ref{def-SDE0}.
\begin{enumerate}[(i)]
\item\label{thm_main_i} The random vector
\(\left(\frac{1}{2T_i}\right)_{i\in V}\)
has law \(\nu^{W,\theta,\eta}_V\),
\item\label{thm_main_ii} Conditionally on \((T_i)_{i\in
  V}\), 
\(((X_{i}(t))_{0\le t\le T_i})_{i\in V}\) are independent
3-dimensional Bessel bridges from \(\theta_{i}\) to \(0\) on time
interval \([0,T_{i}]\).
\end{enumerate}
\end{thm}
Remark that when \(V=\{1\}\) is a single point and \(W_{1,1}=0\), then
\(X_1(t)=Y_1(t)-t\eta_1\) is a drifted Brownian motion with initial value \(\theta_1>0\) and negative drift $-\eta_1$ stopped at 
its first hitting time of 0. Hence, it corresponds to the problem
presented in \eqref{drifted}; in particular \(\eta_{1}=0\) corresponds
to \eqref{Brownian}.  

When \(V=\{1\}\) and \(W_{1,1}>0\), \((Y_1(t))_{t\ge 0}\)
is the solution of the S.D.E.
\begin{equation}\label{drifted-bridge}
dY_1(t)=\mathds{1}_{t<T_1}\left(dB_1(t)-{W_{1,1}\over 1-tW_{1,1}}Y(t)dt\right)
\end{equation}
with initial condition \(Y_{1}(0)=\theta_{1}\). It implies that $Y_1(t)-t\eta_1$ has the law of a drifted Brownian bridge from \(\theta_1\) to \(0\) on time interval
\([0,{1/W_{1,1}}]\) with constant negative drift $-\eta_1$, and stopped at its first hitting of 0. By drifted Brownian bridge from \(\theta_1\) to \(0\) on time interval
\([0,{1/W_{1,1}}]\) with constant negative drift $-\eta_1$ we mean the process $Z_t-t\eta_1$ where $(Z_t)_{t\in [0,1/W_{1,1}]}$ is the Brownian bridge. (It may also be viewed as a Brownian bridge from $\theta_1$ to $-{\eta_1\over W_{1,1}}$ on time interval \([0,{1/W_{1,1}}]\).)
Consequently, \(Y_1(t)\) has the same law as
\((1-tW_{1,1})B_1({{t\over 1-tW_{1,1}}})\) up to time $T_1$, see
e.g. \cite{Revuz-Yor}~p154, and \(T_1\) has the same law as
\({1\over 1+\tau W_{1,1}}\) where \(\tau\) is the first hitting time
of 0 by a Brownian motion with drift \(-\eta_1\). 
Therefore,  \({1\over {1\over T_1}-W_{1,1}}\) follows an Inverse Gaussian law with
parameters \(({\theta_1\over \eta_1}, {\theta_1^2})\), and it is
coherent with the expression of marginal law of \(\beta_i\) in
Corollary~\ref{marginals-beta}.


The next result
shows some "abelianity" of the process, in the sense that times on each
coordinates can be run somehow independently. The first two statements are counterparts of the two statements of Lemma~\ref{lem_marg_cond}.
\begin{thm}[Abelian properties]
\label{thm_abelian}
Let \((X(t))\) be the solution of \eqref{SDE0X}. Denote
\(\beta={1\over 2T}\).
\begin{enumerate}[(i)]
\item\label{thm_abelian-i} {\it (Restriction)} Let \(U\subset
V\). Then, \((X_{U}(t))\) has the same law as the solution of
\(E_{U}^{W_{U,U},\theta_U,\widehat \eta}(X)\), where
\[
\widehat \eta=\eta_{U}+W_{U,U^c}(\theta_{U^c}).
\]
\item\label{thm_abelian-ii}{\it (Conditionning on a subset)} Let \(U\subset V\).
Then, conditionally on \((X_{U}(t))_{t\ge 0}\),
\((X_{U^c}(t))_{t\ge 0}\) has the law of the solutions of the
S.D.E.
\hyperref[SDE0]{\(E_{U^c}^{\widecheck W,\theta_{U^c},\widecheck \eta}(X)\)},
where
\[
\widecheck
W=W_{U^c,U^c}+W_{U^c,U}\left((H_{\beta})_{U,U}\right)^{-1}W_{U,U^c},
\;\;\; \widecheck
\eta=\eta_{U^c}+W_{U^c,U}\left((H_{\beta})_{U,U}\right)^{-1}(\eta_{U}).
\]
\item\label{thm_abelian-iii} {\it (Markov property)} Consider
\(t^0=(t^0_{i})_{i\in V}\in (\R_+)^V\). Denote by
\[ {\mathcal F}^X(t^0)=\sigma\{(X_k(s))_{s\le t^0_k}, \;k\in V\},
\]
the filtration generated by the past of the trajectories before time
\((t^0_k)_{k\in V}\). 
Then, consider for \(t\ge 0\),
\[
\widetilde X(t)=X(t^0+t)\;\left( =(X_i(t^0_i+t))_{i\in V}\right),
\]
the process shifted by times \((t^0_i)_{i\in V}\). (Note that the shift in time is not
necessarily the same for each coordinate).  Conditionally on
\({\mathcal F}^X(t^0)\), the process \((\widetilde X(t))_{t\ge 0}\) has the same law as the solution of the equation 
\hyperref[SDE0X]{\(E_{V}^{\widetilde W^{(t^0)},X(t^0),\widetilde \eta^{(t^0)}}(X)\)}, with
\[
\widetilde W^{(t^0)}=W(K_{t^0\wedge T})^{-1}, \;\;\; \widetilde \eta^{(t^0)}=\eta+\widetilde
W^{(t^0)}((t^0\wedge T)\eta),
\]
where in the second expression, $t^0\wedge T$ denotes the operator of multiplication by $(t_i^0\wedge T_i)$.
In particular, if \(V(t^0)=\{i\in V, \; T_i>t^0_i\}\), conditionally
on \({\mathcal{F}}(t^0)\),
\(\left({1\over T_i-t^0_i}\right)_{i\in V(t^0)}\) has the law
\(\nu_{V(t^0)}^{\widetilde W^{(t^0)}, X(t^0), \widetilde\eta^{(t^0)}}\)
\item\label{thm_abelian-iiii}{\it (Strong Markov property)} Let \(T^0=(T^0_i)_{i\in V}\in (\mathbb{R}_+\cup\{\infty\})^V\) be a ``multi-stopping time'', that is, for all \(t^0\in (\mathbb{R}_+)^V\), the event \(\{T^0\le t^0\}:=\cap_{i\in V}\{T^0_i\le t^0_i\}\) is \(\mathcal{F}^X(t^0)\)-measurable. Denote by
\[\mathcal{F}^X(T^0)=\{A\in \mathcal{F}^X(\infty),\ \forall t^0\in (\mathbb{R}_+)^V,\ A\cap\{T^0\le t^0\}\in \mathcal{F}^X(t^0)\}\]
the filtration of events anterior to \(T^0\). Define for \(t\ge 0\),
\[\widetilde{X}(t)=X(T^0+t)\]
the process shifted at times \((T^0_i)_{i\in V}\). On the event \(\{T^0_i< \infty, \; \forall i\in V\}\), conditionally on  \(T^0\) and \(\mathcal{F}^X(T^0)\), the process \(\widetilde{X}(t)\) has the same law as the solution of the  S.D.E. \hyperref[SDE0X]{\(E_V^{\widetilde{W}^{(T^0)},X(T^0),\widetilde{\eta}^{(T^0)}}(X)\)}, where
\[\widetilde{W}^{(T^0)}=W(K_{T^0\wedge T})^{-1},\ \ \widetilde{\eta}^{(T^0)}=\eta+\widetilde{W}^{(T^0)}((T^0\wedge T)\eta),\]
where in the second expression, $T^0\wedge T$ denotes the operator of multiplication by $(T_i^0\wedge T_i)$.
\end{enumerate}
\end{thm}
\begin{rmk}\label{rmk-prop}
Assertions \eqref{thm_abelian-i} and \eqref{thm_abelian-ii} of the Theorem are 
direct consequence of Theorem~\ref{thm_main1} and Lemma~\ref{lem_marg_cond}. 
The assertion \eqref{thm_abelian-iii} is more involved. The extension to the strong Markov property \eqref{thm_abelian-iiii} follows rather standard arguments. See the proofs in Section~\ref{sec_abelian}.
\end{rmk}

\begin{rmk} In all these statements, the restricted (or conditioned)
process that appears is not in general solution of the S.D.E.  with
the original shifted Brownian motion, but with a different one, which is a
priori not a Brownian motion in the original filtration.
Nevertheless, when all the $t_0^i$ are equal to the same real $s$, then it is the case : $(X(t+s))_{t\ge 0}$ is solution of the S.D.E. with the shifted
Brownian motion $(B(s+t))_{t\ge 0}$, c.f. forthcoming Proposition~\ref{statio}. The result in the latter case is much simpler and is a consequence of a plain computation, whereas the general
case uses the representation of Theorem~\ref{thm_main1}.
\end{rmk}

Note that this allows to identify the law of marginals and conditional marginals.
\begin{coro}
\label{Cor_marginals}
Consider $(X(t))_{t\ge 0}$ solution of \eqref{SDE0X}. Fix \(i_{0}\in V\).
{\hspace*{-3.45em}}\begin{enumerate}[i)]
\item
\label{coro_i}
If \(W_{i_{0},i_{0}}=0\) (resp. \(W_{i_{0},i_{0}}>0\)), the marginal \((X_{i_{0}}(t))_{t\ge 0}\) has the law of a drifted Brownian motion starting
at \(\theta_{i_{0}}\) (resp. drifted Brownian bridge from \(\theta_{i_{0}}\) to 0 on
time interval \([0,{1\over W_{i_{0},i_{0}}}]\), with the meaning given in the discussion of equation \eqref{drifted-bridge}) with constant drift
\[-\widehat\eta_{i_{0}}=-(\eta_{i_{0}}+\sum_{j\neq i_{0}} W_{i_{0},j}\theta_j)\] and stopped
at its first hitting time of 0.
\item
\label{coro_ii}
Conditionally on
\(((X_k(t))_{t\ge 0})_{k\neq i_{0}}\), the process \((X_{i_{0}}(t))_{t\ge 0}\) has the law of
a drifted Brownian bridge from \(\theta_{i_{0}}\) to 0 on time interval
\([0,{1\over \widecheck W_{i_{0},i_{0}}}]\) with constant drift \(-\widecheck\eta_{i_{0}}\) and
stopped at its first hitting time of 0, where, with
\(U=V\setminus\{i_{0}\}\),
\[
\widecheck W_{i_{0},i_{0}}= W_{i_{0},i_{0}}+W_{i_{0},U}( (H_{\beta})_{U,U})^{-1}W_{U,i_{0}},\;\;\;
\widecheck\eta_{i_{0}} =\eta_{i_{0}} +W_{i_{0},U}( (H_{\beta})_{U,U})^{-1}(\eta_U).
\]
\end{enumerate}
\end{coro}
\begin{proof}
Apply Theorem~\ref{thm_abelian}~\eqref{thm_abelian-i} to the case \(U=\{i_{0}\}\) for \eqref{coro_i} and Theorem~\eqref{thm_abelian}~\eqref{thm_abelian-ii} to $U=\{i_0\}^c$ for \eqref{coro_ii}, and the considerations following Theorem~\ref{thm_main1}.
\end{proof}
In particular, it means that the marginal \((X_{i_{0}}(t))_{t\ge 0}\) is a
diffusion process, as well as the (conditional) marginal
\((X_{i_{0}}(t))_{t\ge 0}\) conditioned on
\(((X_k(t))_{t\ge 0})_{\{ k: k\ne i_{0}\}}\).  This Markov property is not
obvious in the initial equation \ref{SDE0X}. Indeed, the process
\((X_{i_{0}}(u))_{u\le s}\) before time \(s\) affects the drifts of
\((X_{\{ k: k\ne i_{0}\}}(u))_{u\le s}\), and so the values
\(X_{\{ k: k\ne i_{0}\}}(s)\), which themselves affect the drift of
\(X_{i_{0}}(s)\).

More generally, there are hidden Markov properties in the restricted
process \((X_U(t))_{t\ge 0}\). Indeed, the law of the future path
\((X_U(t))_{t\ge s}\) only depends on the past of
\((X_{U}(u))_{u\le s}\) through the values of \(X_U(s)\) and
\((s\wedge T)_{U}\). This is not obvious from the initial equation
\ref{SDE0X}.  The same is true for the process
\((X_{U^c}(t))_{t\ge 0}\) conditioned on \((X_U(t))_{t\ge 0}\).

\subsection{Relation with the Vertex Reinforced Jump
  Process}\label{link_VRJP}

Let us describe the VRJP in its "exchangeable" time scale introduced
in \cite{ST15}. We consider the VRJP with a general initial local time, as in \cite{STZ15}, Section~3.1. 
The VRJP, with initial local time $(\theta_i)_{i\in V}$, 
is the self-interacting process
\((Z_t)_{t\ge 0}\) that, conditionally on its past at time \(t\),
jumps from a vertex \(i\) to \(j\) with rate
\[
W_{i,j}{\sqrt{\theta_j+\ell^{Z}_j(t)} \over \sqrt{\theta_i+\ell^{Z}_i(t)}},
\]
where \(\ell^{Z}_j(t)=\int_0^t\mathds{1}_{Z_s=i}ds\) denotes the local
time of \(Z\) at site \(i\). In \cite{ST15}, it was proved that this
process is a mixture of Markov Jump Processes and that the mixing law
can be represented by a marginal of a supersymmetric \(\sigma\)-field
investigated by Disertori, Spencer, Zirnbauer in
\cite{Zirnbauer91,DSZ06,DS10}. In \cite{STZ15}, it was related to the
random potential \(\beta\) of Lemma~\ref{lem_beta}.
\begin{theoA}[\cite{ST15} Theorem 2~\cite{STZ15} Theorem~3]
\label{VRJP}
Let \(\delta\in V\) where \(V\) is finite, and
\(U=V\setminus\{\delta\}\). Let \((\theta_i)_{i\in V}\in (\R_+^*)^V\)
be a positive vector.  Consider \(\beta=(\beta_j)_{j\in V}\) sampled
with distribution \(\nu^{W,\theta}_V\).  Define \( (\psi_j)_{j\in V}\)
as the unique solution of
\[
\begin{cases}
\psi(\delta)=1,\\
H_\beta(\psi)_{|U}=0.
\end{cases}
\]
Then, the VRJP starting at vertex \(\delta\) and initial local times
\((\theta_i)_{i\in V}\) is a mixture of Markov jump processes with
jumping rates
\begin{align}\label{jump_rate}
  \demi W_{i,j}{\psi_j\over \psi_i}.
\end{align}
More precisely, it means that
\[
\P^{\operatorname{VRJP},\theta}_\delta(\cdot)=\int
P^\psi_\delta(\cdot)\nu^{W,\theta}_V(d\beta),
\]
where \(\P^{\operatorname{VRJP},\theta}_\delta\) is the law of the VRJP starting at vertex
\(\delta\) and initial local times \((\theta_i)_{i\in V}\) and
\(P^\psi_\delta\) is the law of the Markov jump process with jumping
rates \eqref{jump_rate} starting at vertex $\delta$.
\end{theoA}

Remark that the random variables \((\beta_j)_{j\in U}\) appear as
asymptotic holding times of the VRJP.  Indeed, let \(N_i(t)\) be the
number of visits of vertex \(i\) by \(Z\) before time \(t\). Then, by
Theorem~\ref{VRJP}, the empirical holding times converge
\(\P^{VRJP,\theta}_\delta\) a.s., i.e. the following limit exists a.s.,
\[
\lim_{t\to\infty} {N_i(t)\over \ell^Z_i(t)}=\demi\sum_{j\sim i}
W_{i,j}{\psi_j\over \psi_i}=\beta_i,\;\;\; \forall i\in U,
\]
and, by Lemma~\ref{lem_marg_cond} \eqref{item-2}, \(\beta_U\) has
law \(\nu_{U}^{W,\theta,\eta}\) where
\(\eta=W_{U,\delta} \theta_\delta\). Moreover, conditionally on
\(\beta_U\), the VRJP is a Markov Jump Process with jump rates given by
\eqref{jump_rate}.

Consider now the
S.D.E. \hyperref[SDE0]{\(E^{W_{U,U},\theta_U,\eta}_{U}(Y)\)} with same
parameters.  From Theorem~\ref{thm_main1}, the law
\(({1\over T_i})_{i\in U}\) coincides with that of
\(\beta_{U}\). Moreover, if we set
\[
\psi_j(\infty):=\lim_{t\to\infty}\psi_j(t), \;\;\; \forall j\in U,
\]
then \(\psi(\infty)=\left((H_{1\over
    2T})_{U,U}\right)^{-1}\eta\). Hence, it means that
\(\psi(\infty)\) coincides with the \(\psi\) of Theorem~\ref{VRJP} if
we identify \(\beta_{U}\) and \({1\over 2T}\). Hence, 
\((\beta_{U},\psi)\) of Theorem~\ref{VRJP} has the same law as
\(({1\over 2T},\psi)\) arising in the
S.D.E. \hyperref[SDE0]{\(E^{W,\theta,\eta}_{U}(Y)\)}.

There are remarkable similarities between Theorem~\ref{thm_main1} and
Theorem~\ref{VRJP}. Firstly, \((\beta_i)_{i\in U}\) are homogeneous to
the inverse of  time, and have same distribution in both
cases. Secondly, in both cases, a type of exchangeability appears in
the sense that, conditionally on the limiting holding times or hitting
times, the processes are simpler : in the case of the VRJP, it
becomes Markov; in the case of the S.D.E., the marginals are
independent and diffusion processes (in fact Bessel bridges).


In Section~\ref{sec_martingale}, we push forward this relation, by
explaining the martingale property that appears in \cite{ST15}, and
the exponential martingale property that extends it in \cite{DMR15},
by Theorem~\ref{thm_main1} and the Abelian properties of
Theorem~\ref{thm_abelian}.

Nevertheless, we do not yet clearly understand the relation between
the VRJP and the S.D.E. \hyperref[SDE0]{\(E^{W,\theta,\eta}_V\)}
beyond these remarks.

\subsection{Organization of the paper}
In Section~\ref{sec_results_distrib}, we prove the properties related
to the distribution $\nu_V^{W,\theta, \eta}$, Lemma~\ref{lem_beta},
Lemma~\ref{lem_marg_cond} and Corollary~\ref{marginals-beta}.  In
Section~\ref{sec_key_formulas}, we present some simple key
computations that are used several times in the proofs.  In
Section~\ref{sec_lemme_SDE}, we prove the results concerning existence
and uniqueness of pathwise solution of the S.D.E.,
Lemma~\ref{def-SDE0}, and state and prove Proposition~\ref{statio} mentioned in Remark~\ref{rmk-prop} above.  Section~\ref{sec_proof_main_thm} is devoted to
the proof of the main Theorem~\ref{thm_main1}.  In
Section~\ref{sec_abelian}, we prove the Abelian properties of
Theorem~\ref{thm_abelian}.  Finally, in Section~\ref{sec_martingale},
we explain the relation between the Abelian properties of
Theorem~\ref{thm_abelian} and the martingale that appears in
\cite{SZ15}.

\section{Proof of the results concerning the distribution
  \(\nu_V^{W,\theta,\eta}\) : Lemma~\ref{lem_beta},
   Lemma~\ref{lem_marg_cond} and Corollary~\ref{marginals-beta}}
\label{sec_results_distrib}

Lemma~\ref{lem_beta} and Lemma~\ref{lem_marg_cond} are proved in
\cite{SZ15} (third arXiv version) in the case \(\theta_i=1\) for all
\(i\in V\), see Lemma~3 and Lemma~4 therein (see also \cite{letac2017multivariate}).  The case of general
$\theta$ can be deduced from the special case $\theta=1$ by a change
of variables. More precisely, setting \(\beta'_i=\theta_i^2\beta_i\),
\(W'_{i,j}= \theta_i\theta_j W_{i,j}\), and
\(\eta'_i=\theta_i \eta_i\), then we have
\[
\left<\theta, H_\beta\theta\right>= \left<1, H'_{\beta'} 1\right>,
\;\; \left<\eta, H_\beta^{-1}\eta\right>=\left<\eta', (H'_{\beta'})^{-1} \eta'\right>,
\;\; \left<\eta,\theta\right>=\left<\eta',1\right>\] where
\(H'_{\beta'}=2\beta'-W'\), so that
\(\beta\sim \nu^{W,\theta,\eta}_V\) if and only if
\(\beta'\sim \nu_V^{W',1,\eta'}\).


Corollary~\ref{marginals-beta} is a direct consequence of the
expression of the Laplace transform. Indeed, under $\nu^{W,\theta,\eta}_V$, the Laplace of the
marginal \(\beta_i-W_{i,i}\) is given for \(\zeta\in\R_+\) by
\[
\int \exp\left(\zeta(\beta_i-\demi W_{i,i})\right)\nu^{W,\theta,\eta}_V(d\beta)={\theta_i\over
  \sqrt{\theta_i^2+\zeta}}\exp\left(-\left(\sqrt{\theta_i^2+\zeta}-\theta_i\right)
  \left(\eta_i +\sum_{j\neq i}W_{i,j}\theta_j\right)\right).
\]
It coincides with the Laplace transform of the inverse of the Inverse
Gaussian density. 
More precisely, by changing the parameter of Inverse Gaussian
distribution, we have
\[
\int_0^\infty \exp\left(-{\zeta\over 2x}\right) \left({\lambda\over
    2\pi x^3}\right)^\demi
\exp\left(-{\lambda\left(x-\mu\right)^2\over 2\mu^2x}\right) dx=
{\sqrt{\lambda}\over \sqrt{\zeta+\lambda}} \exp\left(-\sqrt{{\lambda\over
      \mu^2}}\left(\sqrt{\zeta+\lambda}-\sqrt{\lambda}\right)\right)
\]
It means that the law of \(2\beta_i-W_{i,i}\) coincides with the law of
the inverse of an inverse Gaussian random variable with parameters
\((\lambda,\mu)\) such that \(\lambda=\theta_i^2\) and
\(\sqrt{{\lambda}\over \mu^2}= \eta_i +\sum_{j\neq
  i}W_{i,j}\theta_j\).

\section{Simple key formulas}\label{sec_key_formulas}
Let us start by a remark. If $(t_i)\in
(\R_+)^V$ and  $K_t >0$, then the operator $H^{-1}_{1\over
  2t}$ is well-defined even when some of the
$t_i$'s vanish: indeed, using the identity
$$
H^{-1}_{1\over 2t}= K_{t}^{-1} t,
$$
the right-hand side is perfectly well-defined when
$K_t$ is invertible. In all the sequel, we will implicitly consider
that $H^{-1}_{1\over
  2t}$ is defined by this formula when some of the $t_i$'s vanish.

We prove below some simple formulas that will be key tools in
forthcoming computations.
\begin{lem}\label{Kt+s}
Let \((t^0_i)_{i\in V}\) and \((t^1_i)_{i\in V}\) be vectors in
\(\R_+^V\) such that $K_{t^0+t^1}>0$.
\begin{enumerate}[(i)]
\item
We have,
\begin{eqnarray}\label{lemme1_i}
  K_{t^0+t^1}=\widetilde K_{t^1}K_{t^0},
\end{eqnarray}
with
\[
\widetilde K_{t^1}= \operatorname{Id} - t^1 \widetilde W,\;\hbox{ where, }\;
\widetilde W=WK_{t^0}^{-1}
\]
Hence, we also have, with
$\widetilde H_{1\over 2 t^1}= {1\over t^1}-\widetilde W$, (where \(|H|:=\det H\))
\begin{eqnarray}\label{eq-det} {\left\vert H_{1\over 2(
      t^0+t^1)}\right\vert\over \left\vert \widetilde H_{1\over 2
      t^1}\right\vert}= \left(\prod_{i\in V} {t^1_i\over
    t^0_i+t^1_i}\right) \left\vert K_{1\over 2 t^0}\right\vert
\end{eqnarray}
\item Let
\begin{eqnarray*}
  \widetilde \eta=\eta+\widetilde W (t^0\eta),
\end{eqnarray*}
then,
\begin{eqnarray}\label{eta-tilde}
  \widetilde \eta = (t^0)^{-1} H_{1\over 2t^0}^{-1} \eta.
\end{eqnarray}
and,
\begin{eqnarray}\label{eta-H}
  \left<\widetilde \eta, (\widetilde H_{{1\over 2t^1}})^{-1}\widetilde\eta\right>=\left<\eta , (H_{{1\over 2(t^0+t^1)}})^{-1} \eta\right>- \left<\eta, (H_{1\over 2 t^0})^{-1}\eta\right>
\end{eqnarray}
\end{enumerate}
\end{lem}
\begin{rmk}
One should not confound the \(\widetilde{W}\) in Lemma~\ref{Kt+s} (which is deterministic) with the \(\widetilde{W}^{(t^0)}\) in Theorem~\ref{thm_abelian}, which should be consider as a process.
\end{rmk}
\begin{proof}(i) We can write
\[
K_{t^0+t^1}=K_{t^0}-t^1W=(\operatorname{Id}-t^1W K^{-1}_{t^0})K_{t^0}=
\widetilde K_{t^1}K_{t^0}.
\]

(ii) Formula \eqref{eta-tilde} follows from
$$
\widetilde \eta = (t^0)^{-1}(\operatorname{Id}+t^0 W K_{t^0}^{-1})t^0\eta
= (t^0)^{-1} K_{t^0}^{-1} t^0\eta=(t^0)^{-1} H_{{1\over 2t^0}}^{-1}
\eta
$$
Turning to Formula \eqref{eta-H}, using \eqref{lemme1_i}, we have
\begin{eqnarray*}
  K_{t^0+t^1}^{-1}=K_{t^0}^{-1}\widetilde{K}_{t^1}^{-1}
\end{eqnarray*}
and
\begin{equation}
\label{eq-tilhtilbeta-1}
\begin{aligned}
\widetilde H_{{1\over 2t^1}}^{-1} &= K_{t^0}K_{t^0+t^1}^{-1}t^1
\\
&= K_{t^0}K_{t^0+t^1}^{-1}(t^0+t^1)({1\over t^0}-{1\over t^0+t^1})t^0
\\
&= t^0H_{1\over 2t^0}H_{1\over 2(t^0+t^1)}^{-1}({1\over t^0}-{1\over
  t^0+t^1})t^0
\\
&= t^0H_{1\over 2t^0}H_{1\over 2(t^0+t^1)}^{-1}(H_{1\over 2
  t^0}-H_{1\over 2(t^0+t^1)})t^0
\\
&= t^0H_{1\over 2t^0}H_{1\over 2(t^0+t^1)}^{-1}H_{1\over 2 t^0} t^0-
t^0H_{1\over 2t^0} t^0
\end{aligned}
\end{equation}
Now, \eqref{eta-tilde} implies
$$
\widetilde H_{1\over 2 t^1} \widetilde \eta = t^0H_{1\over 2t^0}H_{1\over
  2(t^0+t^1)}^{-1}\eta - t^0 \eta.
$$
Since $H_{1\over 2t^0}$ is symmetric, we get \eqref{eta-H} by \eqref{eta-tilde}.
\end{proof}

\section{Proof of basic properties of the
  S.D.E. \hyperref[SDE0]{\(E^{W,\theta,\eta}_V\)} :
proof of  Lemma~\ref{def-SDE0}.
  }
  \label{sec_lemme_SDE}

Remark that \eqref{lem_i} and \eqref{lem_ii} of Lemma~\ref{def-SDE0}
are equivalent since \(dX(t)=dY(t)-\eta dt\). In order to prove the
existence and uniqueness of the pathwise solution of \ref{SDE0} (or
equivalently \ref{SDE0X}), we first consider a non stopped version of
the S.D.E. \eqref{SDE0}, for which the existence and uniqueness is
simpler.
\begin{lem}\label{lem-finite-lc}
Let \((\theta_i)_{i\in V}\in \R_+^V\).  Let \(h>0\) be the smallest
positive real such that \(\det(K_h)=0\).  Then, the following S.D.E. is
well-defined on time interval \([0,h)\) and has a unique pathwise
solution
\begin{gather}\label{SDE-intermediate}
\widetilde{Y}_{i}(t)= \theta_i+
B_i(t)-\int_0^t (WK_s^{-1} \widetilde{Y}(s))_ids\ \ \ \ \forall i\in V.
\end{gather}
Moreover, there exists a time \(\tau<h\) such that
\(\widetilde{Y}_i(\tau)=\tau \eta_i\) for some vertex \(i\in V\).
\end{lem}
\begin{proof}
As \(WK_t^{-1}\) is bounded on time interval \([0,h-\epsilon)\) for
all \(\epsilon>0\), it is a linear S.D.E with bounded coefficients
there is a unique pathwise solution, with continuous simple paths, by standard existence and uniqueness theorems on S.D.E.

To see the existence of \(\tau\), we can define \((Z_{t})_{t\ge 0}\)
by
\[(h-t)Z_{i}(\frac{t}{h-t})=\widetilde{Y}_{i}(t), \ \ \forall i\in V\] and
write \eqref{SDE-intermediate} as
\[(h-t)Z_{i}(\frac{t}{h-t})=\theta_{i}+B_{i}(t)-\int_{0}^{t}\left[
  WK_s^{-1}
  (h-s)Z(\frac{s}{h-s}) \right]_{i}ds.\] By time change
\(u=\frac{t}{h-t}\), the S.D.E. is written in the following equivalent
form
\[\frac{1}{u+1}Z_{i}(u)=\frac{\theta_{i}}{h}+\frac{1}{h}B_{i}(\frac{hu}{u+1})-\int_{0}^{u}\left[
  WK_{\frac{hv}{v+1}}^{-1}
  Z(v) \right]_{i}\frac{1}{(v+1)^{2}}dv.\] That is
\begin{align}\label{dZiu}
  dZ_{i}(u)=\frac{1}{\sqrt{h}}d\widetilde B_{i}(u)+\frac{1}{u+1}\left( \left[\operatorname{Id}- WK_{\frac{hv}{v+1}}^{-1}
  \right]Z(u) \right)_{i}du.
\end{align}
where $(\widetilde B_i(t))_{i\in V}$ is a $N$-dimensional Brownian motion. As \(t\to h\), we have \(u\to\infty\), and there exists $\tau<h$ such
that \(\widetilde{Y}_{i}(\tau)=\tau \eta_{i}\) if and only if there exists
$\tau'\in \R_+$ such that \(Z_{i}(\tau')=\tau' \eta_{i}\).  Assume by
contradiction that none of these \(Z_{i}\) reach the lines
\(y=\eta_{i}x\), in particular, they are all positive.  We use that
$K_s^{-1}$ has positive coefficients and that
$\lim_{s\to h} \min_{i,j} (K_s^{-1})_{i,j}=+\infty$, which implies
that for $u$ large enough
$ (\operatorname{Id}- WK_{\frac{hv}{v+1}}^{-1}) $ has negative
coefficients, hence the drift term in \eqref{dZiu} is negative. This
implies that $Z_{i}(u)$ given by \eqref{dZiu} is stochastically
bounded from above by a Brownian motion, at least for $u$ large enough. Hence, the processes $(Z_i(u))_{u\ge 0}$ reach 0 in finite time,
which leads to a contradiction.
\end{proof}
\begin{proof}[Proof of Lemma~\ref{def-SDE0}~\eqref{lem_i}]
We prove it by recurrence on the size of \(V\).
We will gradually define \(Y(t)\), solution to the equation
\eqref{SDE0} and \(X(t)=Y(t)-t\eta\).  Consider
\[
\tau=\inf\{t\ge 0, \; \exists i\in V\hbox{ such that } X_i(t)=0\}
\]
and denote by \(i_0\) the vertex in \(V\) such that
\(X_{i_0}(\tau)=0\).  Up to time \(\tau\), the equation \eqref{SDE0}
is equivalent to the equation \eqref{SDE-intermediate}, hence the
equation \eqref{SDE0} is well-defined and has unique pathwise solution
up to time \(\tau\) and \(\tau<\infty \) a.s.. Moreover,
\(T_{i_0}=\tau\).  Now we set \(U=\{i_0\}^c\) and
\[(\widetilde T_i)_{i\in V}=(T_i-\tau)_{i\in V}\]
\[
\widetilde W=WK^{-1}_{\tau}, \;\;\; \widetilde K_{s} = \operatorname{Id} -s
\widetilde W, \;\;\; \widetilde \eta= \eta+\widetilde W(\tau \eta).
\]
and use that, by \eqref{lemme1_i} applied to $t^0_i=\tau$ for all $i$,
and $t^1=s\wedge \widetilde T$,
\[
K_{(\tau+s)\wedge T}^{-1}= K^{-1}_{\tau}\widetilde K^{-1}_{s\wedge \widetilde
  T}.
\]
We set
\[ \widetilde X(s)=X(\tau+s), \;\;\; \widetilde B(s)=B(\tau+s).\] Hence, we
have that
\[ (\tau+s)\wedge T= \tau + s\wedge \widetilde T, \;\;\; WK^{-1}_{(\tau+s)\wedge T}=\widetilde W \widetilde K^{-1}_{s\wedge T}\] 
and
after time \(\tau\), \((X_{\tau+t})_{t\ge 0}\) is solution of
\hyperref[SDE0]{\(E^{W,\theta,\eta}_{V}(X)\)} if and only if
$\widetilde X(s)$ is solution of
\begin{align}\label{eq-tilde-X}
  d\widetilde X(s) = \mathds{1}_{s<\widetilde T} d\widetilde B(s)+\mathds{1}_{s<\widetilde T}\left(\widetilde W\widetilde K_{s\wedge \widetilde T}^{-1}
  \left(\widetilde X(s)+\tau \eta+(s\wedge \widetilde T)\eta\right)+\eta\right)ds.
\end{align}
Using that,
\begin{align*}
  &\widetilde W \widetilde K_{s\wedge\widetilde T}^{-1}\left(\widetilde X(s)+\tau\eta+ (s\wedge \widetilde T)\eta\right)
  \\
  =&
     \widetilde W \widetilde K_{s\wedge\widetilde T}^{-1}\left(\widetilde X(s)+\widetilde K_{s\wedge \widetilde T}(\tau\eta)+(s\wedge \widetilde T)\widetilde W(\tau \eta)+ (s\wedge \widetilde T)\eta\right)
  \\
  =&
     \widetilde W \widetilde K_{s\wedge\widetilde T}^{-1}\left(\widetilde X(s)+(s\wedge \widetilde T)\widetilde\eta\right)+\widetilde W(\tau\eta)
\end{align*}
we see that \eqref{eq-tilde-X} is equivalent to the fact that
$\widetilde X$ is solution of
\hyperref[SDE0X]{\(E_{V}^{\widetilde W,X(\tau),\widetilde \eta}(X)\)}. Since,
$X_{i_0}(\tau)=0$ is it equivalent to the fact that
$\widetilde X_{U}$ is solution of
\hyperref[SDE0X]{\(E_{U}^{\widetilde W,X(\tau),\widetilde \eta}(X)\)}.
Hence, we conclude by the recurrence hypothesis applied to $U$,
which implies that
\hyperref[SDE0X]{\(E_{U}^{\widetilde W,X(\tau),\widetilde \eta}(X)\)} has
a unique pathwise solution.
\end{proof}
\begin{proof}[Proof of Lemma~\ref{def-SDE0}~\eqref{lem_iii}]
Remark first that
\[ {\partial\over\partial t} K_{t\wedge T}^{-1}=K_{t\wedge T}^{-1}
\mathds{1}_{t<T} W K_{t\wedge T}^{-1}.
\]
Differentiating \(\psi(t)=K_{t\wedge T}^{-1}(Y(t))\), we get, 
\begin{align*}
  d\psi_i(t)&=
              (K_{t\wedge T}^{-1}(dY(t)))_i+\left(K_{t\wedge T}^{-1} \mathds{1}_{t<T} WK_{t\wedge T}^{-1}(Y(t))\right)_i dt
  \\
            &=(K_{t\wedge T}^{-1}( \mathds{1}_{t<T}dB(t)))_i
\end{align*}
Moreover, the quadratic variation of \(\psi_i(t)\) and \(\psi_j(t)\)
is given by
\begin{align*}
  \left<\psi_i, \psi_j\right>_{t}&= \sum_{l\in V} \int_0^t (K_{s\wedge T}^{-1})_{i,l} \mathds{1}_{s<T_l} (K_{s\wedge T}^{-1})_{j,l} ds
  \\
                                 &= \sum_{l\in V} \int_0^t (H_{{1 \over 2  (s\wedge T)}}^{-1})_{i,l} \left({1\over s\wedge T_l}\right)^2\mathds{1}_{s<T_l} (H_{{1\over 2(s\wedge T)}}^{-1})_{l,j} ds
  \\
                                 &=\int_0^t {\partial\over\partial s}(H_{{1\over 2(s\wedge T)}}^{-1})_{i,j} ds
  \\
                                 &=
                                   (H_{{1\over 2(t\wedge T)}}^{-1})_{i,j}
\end{align*}
where in the second equality, we used \(H_\beta\) is a symmetric
matrix and
\(H^{-1}_{1\over 2(t\wedge T)}=K_{t\wedge T}^{-1}(t\wedge T)\), and so
that
\(H^{-1}_{1\over 2(t\wedge T)}=(t\wedge T) (K_{t\wedge
  T}^{-1})^t\). In the last equality we used that
\(H^{-1}_{1\over 2(t\wedge T)}\) is well defined and null for \(t=0\).
\end{proof}
\section{Stationarity property}
\begin{prop}[Stationarity]\label{statio}
If \((X(t))_{t\ge 0}\) is the solution of
\eqref{SDE0X} and \(s\ge 0\), then \((X(t+s))_{t\ge 0}\) is solution
of the S.D.E.
\hyperref[SDE0X]{\(E^{\widetilde W^{(s)}, X(s), \widetilde \eta^{(s)}}_{V}(X)\)} directed by the
shifted brownian motion \((B(t+s))_{t\ge 0}\), and with
\[
\widetilde W^{(s)}= WK_{s\wedge T}^{-1},\;\; \widetilde \eta^{(s)}=\eta+\widetilde W^{(s)} ((s\wedge
T)\eta),
\]
\end{prop}
\begin{rmk}
Proposition~\ref{statio} corresponds to Theorem~\ref{thm_abelian}~\eqref{thm_abelian-iii} in the case where all the coordinates of $(t^0_i)$ are equal to $s$, except that in this case the equation is directed by the shifted Brownian motion, which is not the case when coordinates are not all equal. The proof in this case is based on elementary computations and do not rely on the representation given in Theorem~\ref{thm_main1}. The result can then be interpreted as a dynamic evolution of the parameters along the trajectory: conditioned on the past, the futur of the trajectory is in the same family of S.D.E with deformed parameters.
\end{rmk}
\begin{proof}[Proof of Proposition~\ref{statio}]

Set \((\widetilde X(t))_{t\ge 0}:= (X(t+s))_{t\ge 0}\),
\((\widetilde B(t))_{t\ge 0}:= (B(t+s))_{t\ge 0}\), and
\(\widetilde T^{(s)}=T-s\wedge T\).  Remark that by Lemma~\ref{Kt+s}
$$
(s+t)\wedge T= s\wedge T+ t\wedge \widetilde T^{(s)}, \;\;\; WK_{(s+t)\wedge T}= \widetilde W^{(s)} \widetilde K_{t\wedge \widetilde T}^{(s)}
$$
with $\widetilde W^{(s)}$ defined in Proposition~\ref{statio} and \(\widetilde K_{t\wedge \widetilde T}^{(s)}= \operatorname{Id}-(t\wedge \widetilde T^{(s)})\widetilde W^{(s)}\).
The
S.D.E. \hyperref[SDE0X]{\(E_{V}^{W,\theta,\eta}(X)\)} after time \(s\)
is thus equivalent to
\[
d \widetilde X_i(t)= \mathds{1}_{t<\widetilde T^{(s)}_i}d\widetilde B_i(t)-
\mathds{1}_{t<\widetilde T^{(s)}_i} \left(\widetilde W^{(s)} (\widetilde K_{t\wedge\widetilde T}^{(s)})^{-1}\left(\widetilde X(t)+(s\wedge T)\eta+ (t\wedge \widetilde T^{(s)})\eta\right) +\eta\right)_idt,\;\;\; \forall i\in V,
\]
By Lemma~\ref{Kt+s}, we have that
\begin{align*}
  &\widetilde W^{(s)} (\widetilde K_{t\wedge\widetilde T}^{(s)})^{-1}\left(\widetilde X(t)+(s\wedge T)\eta+ (t\wedge \widetilde T^{(s)})\eta\right)
  \\
  =&
     \widetilde W^{(s)} (\widetilde K_{t\wedge\widetilde T}^{(s)})^{-1}\left(\widetilde X(t)+\widetilde K^{(s)}_{t\wedge \widetilde T}((s\wedge T)\eta)+(t\wedge \widetilde T^{(s)})\widetilde W^{(s)}((s\wedge T)\eta)+ (t\wedge \widetilde T^{(s)})\eta\right)
  \\
  =&
     \widetilde W^{(s)} (\widetilde K_{t\wedge\widetilde T}^{(s)})^{-1}\left(\widetilde X(t)+(t\wedge \widetilde T^{(s)})\widetilde\eta^{(s)}\right)+\widetilde W^{(s)}(s\wedge T)\eta
\end{align*}
Hence, \(\widetilde X(t)\) is solution of
\[
d\widetilde X(t)= \mathds{1}_{t<\widetilde T^{(s)}_i}d\widetilde B_i(t)-
\mathds{1}_{t<\widetilde T^{(s)}_i} \left(\widetilde W^{(s)} (\widetilde K_{t\wedge\widetilde
    T}^{(s)})^{-1}\left(\widetilde X(t)+ (t\wedge \widetilde T^{(s)})\widetilde\eta^{(s)}\right)
  +\widetilde \eta^{(s)}\right)_i dt,\;\;\; \forall i\in V,
\]
Since, \(\widetilde X(0)=X(s)\), we have the result.
\end{proof}

\section{Proof of Theorem \ref{thm_main1}}\label{sec_proof_main_thm}
We provide below a convincing but incomplete argument for the proof of Theorem~\ref{thm_main1}. We do not know yet how to turn this argument into a rigourous alternative proof, even though we think that it should be possible.
The rigorous proof is given in Section~\ref{rigorous_proof}.
\subsection{A convincing but incomplete argument for
  Theorem~\ref{thm_main1}~\eqref{thm_main_i}}
Let \(\lambda\in \R_+^V\) be a non negative vector on \(V\). As
\[
\exp\left(-\left<\eta, H_\beta^{-1} \lambda\right>-\demi\left<\lambda,
    H_\beta^{-1} \lambda\right>\right)\nu_V^{W,\theta,\eta}
=\exp\left(-\left<\lambda,\theta\right>\right)\nu_V^{W,\theta,\eta+\lambda},
\]
we have,
\begin{align}
  \label{equation_lambda}
  \int
  \exp\left(-\left<\eta, H_\beta^{-1} \lambda\right>-\demi\left<\lambda, H_\beta^{-1} \lambda\right>\right)\nu_V^{W,\theta,\eta}(d\beta)
  =
  \exp\left(-\left<\lambda,\theta\right>\right).
\end{align}
On the other hand, consider \(Y(t)\), solution of \ref{SDE0}, and the
associated processes \((X(t))\), \((\psi(t))\).  By
Lemma~\ref{def-SDE0} and \cite{Revuz-Yor} proposition 3.4 p 148, we
know that
\[
\exp\left(-\left<\lambda, \psi(t)\right>-\demi\left<\lambda,
    H_{{1\over 2( t\wedge T)}}^{-1} \lambda\right>\right),
\]
is a continuous martingale, dominated by 1. Moreover, we have that
\(X(t)\to 0\), a.s., when \(t\to \infty\), hence, a.s.,
\[
\lim_{t\to\infty} \psi(t)= K^{-1}_T(T\eta)=H_{{1\over 2 T}}^{-1} \eta.
\]
By dominated convergence theorem, it implies that
\[
\E\left(\exp\left(-\left<\lambda, H_{{1\over 2 T}}^{-1}\eta
    \right>-\demi\left<\lambda, H_{{1\over 2 T}}^{-1}
      \lambda\right>\right)\right)= \exp\left(-\left<\lambda, \psi(0)
  \right>\right)= \exp\left(-\left<\lambda, \theta
  \right>\right).
\]
Hence, it implies that both \(\beta\) under \(\nu^{W,\theta,\eta}_V\)
and \({1\over 2T}\) obtained from
\hyperref[SDE0]{\(E^{W,\theta,\eta}_V\)} satisfy the same functional
identity \eqref{equation_lambda}. Note that the dimension of the space of
variables $(\lambda_i)_{i\in V}$ and of the random variables
$(\beta_i)_{i\in V}$ are the same.
Nevertheless, it is not clear wether the functional identity
\eqref{equation_lambda} characterizes the distribution
\(\nu_V^{W,\theta,\eta}\), at least we have no proof of this fact.  If
such an argument were available, it would imply
Theorem~\ref{thm_main1}~\eqref{thm_main_i} also : indeed, using the
stationarity of the equation, Proposition~\ref{statio}, it would be
possible to deduce Theorem~\ref{thm_main1}~\eqref{thm_main_ii} by
enlargement of filtration (see \cite{jeulin2006grossissements}). We do
not give the detail of the argument here since the first part of the
proof is missing.

\subsection{Proof}\label{rigorous_proof}
Even if it is not obvious at first sight since the context is very different, the strategy of the proof of Theorem~\ref{thm_main1} is quite in the spirit
of the proof of Theorem 2, ii) of \cite{ST15} : we start from the
mixture of Bessel processes and we prove that this mixture has the
same law as the solutions of the S.D.E. \ref{SDE0X}. We use in a
crucial way the fact that the law $\nu^{W,\theta, \eta}_V$ is a
probability density with explicit normalizing constant. 

\subsubsection{The classical statement for \(N=1\)}
We denote by \({\bf W}=C(\R_+,\R)\) the Wiener space.  For
\(\theta>0\), we denote by \(\P_\theta\) the law of \(X_{t\wedge T}\)
where \(X_t=\theta+B_t\) and \(B_t\) is a standard brownian motion and
\(T=\inf\{t\ge 0, \; X_t=0\}\) is the first hitting time of \(0\).  We
denote by \(\B_{\theta,0}^{3,T}\) the law of the 3-dimensional Bessel
Bridge from \(\theta>0\) to \(0\) on time interval \([0,T]\), as
defined in \cite{Revuz-Yor}, section XI-3. We always consider that the
Bessel bridge is extended to time interval \(\R_+\), with constant
value equal to 0 after time \(T\).  As mentioned in the introduction
it is known (see \cite{MR0350881}~\cite{Revuz-Yor}, p317), that,
under \(\P_{\theta}\), \({1\over 2T}\) has the law
Gamma\((\demi,{\theta^{2}\over 4})\) and that, conditionally on \(T\),
\((X_t)_{t\ge 0}\) has law \(\B_{\theta,0}^{3,T}\). Otherwise stated
it means that the following equality of probabilities holds on the Wiener space ${\bf W}$:
\begin{align}\label{one-dim}
  \P_{\theta}(\cdot )= \int_0^\infty \B_{\theta,0}^{3,T}(\cdot) {1\over \sqrt{2\pi}}{\theta\over T^{3/2}}e^{-{\theta^2\over 2T}} dT
\end{align}


\subsubsection{Proof of Theorem \ref{thm_main1} (i) and (ii)}
We use the formulation of Lemma~\ref{def-SDE0}~(\ref{lem_ii}), and we
will prove that if \((X_i(t))_{i\in V}\) satisfies \eqref{SDE0X}, then
\(\beta:={1\over 2T}\) is distributed as \(\nu_V^{W,\theta,\eta}\) and
conditionally on \(T\), the coordinates \((X_i(t))_{t\ge 0}\) are
independent 3-dimensional Bessel bridges from \(\theta_i\) to \(0\) on time
interval \([0,T_i]\).

Recall that \(V=\{1,\ldots, N\}\), and denote by
\({\bf W}_V=C(\R_+,\R^V)\) the \(N\)-dimensional Wiener space and
\((X(t))_{t\ge 0}\) the canonical process.  For
\(\theta=(\theta_i)_{i\in V}\in \R_+^V\), we set
\[
\P_{V, \theta}=\otimes_{i\in V} \P_{\theta_i},
\]
the probability on \({\bf W}_V\) such that \((X_i(t))_{i\in V}\) are
\(N\) independent Brownian motions starting at positions
\((\theta_i)\) and stopped at their first hitting times of \(0\).  The
assertions of Theorem~\ref{thm_main1}~\eqref{thm_main_i}
and~\eqref{thm_main_ii} are equivalent to the fact that the law of the
solution of the S.D.E. \eqref{SDE0X} is a mixture of independent
Bessel bridges \(\B^{3,{1\over 2\beta_i}}_{\theta_i,0}\) where
\(\beta\) is a random vector with distribution
\(\nu_V^{W,\theta,\eta}\). Otherwise stated, it means that the
probability distribution $\overline\P^{W,\theta,\eta}_{V}$ defined by
\[
\overline\P^{W,\theta,\eta}_{V}(\cdot):= \int \left( \otimes_{i\in V}
  \B_{\theta_i,0}^{3,{1\over 2 \beta_i}}\right)(\cdot)
\nu_{V}^{W,\theta,\eta}(d\beta),
\]
is the law of the solution of the S.D.E. \eqref{SDE0X}.  The strategy
is now to write the Radon-Nikodym derivative of
\(\overline\P^{W,\theta,\eta}_{V}\) with respect to \(\P_{V,\theta}\)
as an exponential martingale, and then to apply Girsanov's theorem.

In the sequel, we adopt the following notations:
\[
T:={1\over 2\beta},\;\;{\hbox{ so that }}\;\; H_\beta={1\over T}K_T.
\]
From \eqref{one-dim}, it is clear that
\(\overline\P^{W,\theta,\eta}_{V}\) is absolutely continuous with
respect to \( \P_{V, \theta}\), and changing from variables \(\beta\)
to \(T\) in \(\nu_{V}^{W,\theta,\eta}(d\beta)\), we get that
\begin{equation}
\label{chgt-density}
\begin{aligned}
d{\overline\P^{W,\theta,\eta}_{V}\over \P_{V,
    \theta}}&=\mathds{1}_{H_{\frac{1}{2T}}>0}\cdot
e^{-\demi\left<\theta, H_{1\over 2T} \theta\right>+\demi\left<\theta,
    {1\over T}\theta\right>-\demi\left<\eta,(H_{1\over
      2T})^{-1}\eta\right>+\left<\eta,\theta\right>}{\prod_{i\in V}
  T^{-1/2}_i\over \sqrt{\vert H_{1\over 2T}\vert}}.
\\
&=\mathds{1}_{H_{\frac{1}{2T}}>0}\cdot \exp\left({\demi\left<\theta, W
    \theta\right>-\demi\left<\eta, K_T^{-1}
    T\eta\right>+\left<\eta,\theta\right>}\right){1\over \sqrt{\vert
    K_{T}\vert}}.
\end{aligned}
\end{equation}
Let \(t>0\),
define
\[
\begin{cases}
V(t):=\{i\in V, \; T_i>t\},\\
\beta^{(t)}:={1\over 2 (t\wedge T)}\\
\widetilde W^{(t)}:=WK_{t\wedge T}^{-1}=W+WK_{t\wedge T}^{-1}(t\wedge T) W
\\
\widetilde \eta^{(t)}:=\eta+\widetilde W^{(t)}(t\wedge T)\eta
\end{cases}
\]
where the third equality comes from the fact that
\(K_{t\wedge T}^{-1}=\operatorname{Id} +(t\wedge T)WK_{t\wedge
  T}^{-1}\).  Note that
\(\widetilde W^{(t)}\) is symmetric since
\(K_{t\wedge T}^{-1}(t\wedge T)=H_{{1\over 2(t\wedge T)}}^{-1}\). We
also set,
\[
\begin{cases}
\widetilde T^{(t)}:= T -t\wedge T,\\  \widetilde \beta^{(t)}:={1\over 2\widetilde T^{(t)}},\\  
\widetilde K^{(t)}_{\widetilde T}:= \operatorname{Id} -\widetilde T^{(t)}\widetilde W^{(t)},\\
\widetilde H_{\widetilde \beta}^{(t)}:=2\widetilde \beta^{(t)}-\widetilde W^{(t)}={1\over \widetilde T^{(t)}}\widetilde K^{(t)}_{\widetilde T},\\
\end{cases}
\]
Note that \((\widetilde H_{\widetilde \beta}^{(t)})^{-1}\) is well defined for
all \(t\) using
\((\widetilde H_{\widetilde \beta}^{(t)})^{-1}=(\widetilde K^{(t)}_{\widetilde T})^{-1}
\widetilde T^{(t)}\), see beginning of section \ref{sec_key_formulas}.  By
Equation (\ref{eta-tilde}) applied with $t^0=t\wedge T$ and $t^1=\widetilde T^{(t)}$,
we get that
\begin{equation}
\label{tilde_eta}
\begin{aligned}
\widetilde
 \eta^{(t)}
=(t\wedge T)^{-1}H_{\beta^{(t)}}^{-1}\eta .
\end{aligned}
\end{equation}
We first prove the following lemma.
\begin{lem}\label{calcul-Ito}
Let
\begin{eqnarray*}
  M_t=
  \exp\left(-\demi\left<X(t), \widetilde W^{(t)} X(t)\right>+\demi\left<\widetilde \eta^{(t)}, (\widetilde H_{\widetilde \beta}^{(t)})^{-1}\widetilde \eta^{(t)}\right>-\left<\widetilde \eta^{(t)},X(t)\right>\right)
  {\sqrt{\vert \widetilde K^{(t)}_{\widetilde T}\vert}}.
\end{eqnarray*}
Under \(\P_{V,\theta}\), we have
\begin{align}
  {M_t\over M_0}=\exp\left(-\int_0^t \left<W\psi(s)+\eta, dX_s\right>-\demi\int_0^t\left<W\psi(s)+\eta, \mathds{1}_{s<T}(W\psi(s)+\eta)\right>ds\right)
\end{align}
with
\[
\psi(t)=K_{t\wedge T}^{-1}(X(t)+(t\wedge T)\eta).
\]
\end{lem}
\begin{proof}[Proof of Lemma~\ref{calcul-Ito}]
We will compute the It\^o derivative of \(\ln M_{t}\), the following
formulae will be used several times
\begin{align}\label{partial_K_t}
  {\partial\over\partial t} K_{t\wedge T}=- \mathds{1}_{t<T} W, \;\;\;  {\partial\over\partial t} K_{t\wedge T}^{-1}= K_{t\wedge T}^{-1}\mathds{1}_{t<T} WK_{t\wedge T}^{-1}, \;\;
  {\partial\over\partial t}\widetilde W^{(t)}=\widetilde W^{(t)}\mathds{1}_{t<T}\widetilde W^{(t)}.
\end{align}
\begin{align}\label{partial_H-1}
  {\partial \over \partial t} H_{\beta^{(t)}}^{-1}&=H_{\beta^{(t)}}^{-1}\mathds{1}_{t<T}\left({1\over t\wedge T}\right)^2 H_{\beta^{(t)}}^{-1}
\end{align}
By \eqref{partial_K_t} and It\^o formula, we have
\begin{eqnarray}
  \nonumber
  &&d\left<X(t), \widetilde W^{(t)} X(t)\right>\\
  &=&
      \label{form_1} 2\left<dX(t), \widetilde W^{(t)}X(t)\right>+\left<\widetilde W^{(t)} X(t), \mathds{1}_{t<T}  \widetilde W^{(t)}X(t)\right>dt+\trace(\widetilde W^{(t)}\mathds{1}_{t<T}) dt
\end{eqnarray}
where in the second term we used that the operator \(\widetilde W^{(t)}\) is
symmetric.
By \eqref{eta-H} of Lemma~\ref{Kt+s} applied to $t^0=t\wedge T$ and
$t^1=\widetilde T^{(t)}$, we get
\begin{align*}
  \left<\widetilde \eta^{(t)}, (\widetilde H_{\widetilde \beta}^{(t)})^{-1}\widetilde \eta^{(t)}\right>=\left<\eta , (H_\beta)^{-1} \eta\right>- \left<\eta, (H_{\beta^{(t)}})^{-1}\eta\right>
\end{align*}
Using \eqref{partial_H-1} and \eqref{tilde_eta}, it implies,
\begin{eqnarray}
  \label{form_2}
  d\left<\widetilde \eta^{(t)}, (\widetilde H_{\widetilde \beta}^{(t)})^{-1}\widetilde \eta^{(t)}\right>=
        -\left<\widetilde \eta^{(t)}, \mathds{1}_{t<T}\widetilde \eta^{(t)}\right>dt.
\end{eqnarray}
We have also
\begin{eqnarray*} {\partial\over \partial t} \widetilde \eta^{(t)}= \widetilde W^{(t)}\mathds{1}_{t<T}\eta+\widetilde W^{(t)}\mathds{1}_{t<T}\widetilde W^{(t)}(t\wedge T)\eta= \widetilde W^{(t)}\mathds{1}_{t<T} \widetilde \eta^{(t)}.
\end{eqnarray*}
Hence,
\begin{eqnarray}\label{form_3}
  d\left<\widetilde \eta^{(t)}, X(t)\right>=\left<\widetilde \eta^{(t)}, dX(t)\right>+\left<\widetilde \eta^{(t)}, \mathds{1}_{t<T}\widetilde W^{(t)}X_t\right>dt.
\end{eqnarray}
Finally, using \eqref{eta-H} of Lemma~\ref{lemme1_i} applied to
$t^0=t\wedge T$ and $t^1=\widetilde T^{(t)}$, we get
\begin{eqnarray}\label{product-K2}
  K_T^{-1}=K_{t\wedge T}^{-1}(\widetilde K^{(t)}_{\widetilde T})^{-1}
\end{eqnarray}
which implies by \eqref{partial_K_t},
\begin{equation}\label{form_4} {\partial\over \partial t}\ln\vert
\widetilde K^{(t)}_{\widetilde T}\vert= -{\partial\over \partial t} \ln\vert
K_{t\wedge T}\vert=-\trace(\mathds{1}_{t<T} WK^{-1}_{t\wedge
  T})=-\trace(\mathds{1}_{t<T} \widetilde W^{(t)}).
\end{equation}
Combining~\eqref{form_1},~\eqref{form_2},~\eqref{form_3},
and~\eqref{form_4}, we get using that
\(W\psi(t)+\eta= \widetilde W^{(t)}X(t)+\widetilde \eta^{(t)}\),
\begin{eqnarray*}
  d\ln M_t
  &=& - \left<dX(t), \widetilde W^{(t)}X(t)+\widetilde \eta^{(t)}\right>-\demi\left<\widetilde W^{(t)} X(t), \mathds{1}_{t<T}  \widetilde W^{(t)}X(t)\right>dt
  \\
  &&-\demi\left<\widetilde \eta^{(t)}, \mathds{1}_{t<T}\widetilde \eta^{(t)}\right>dt
     -\left<\widetilde \eta^{(t)}, \mathds{1}_{t<T}\widetilde W^{(t)}X_t\right>dt
  \\
  &=&
      -\left<W\psi(t)+\eta, dX_t\right>-\demi\left<W\psi(t)+\eta, \mathds{1}_{t<T}(W\psi(t)+\eta)\right>dt
\end{eqnarray*}
Consider now a positive measurable test function
\(\phi((X_s)_{s\le t})\). Denote by \(\overline\E^{W,\theta,\eta}_V\),
(resp. \(\E_{V,\theta}\)), the expectation with respect to
\(\overline\P^{W,\theta,\eta}_V\), (resp. \(\P_{V,\theta}\)).  We
have, by \eqref{chgt-density},
\begin{align*}
  &\overline\E^{W,\theta,\eta}_{V}\left(\phi((X_s)_{s\le t})\right))
  \\
  &=
    \E_{V,\theta}\left(\phi((X_s)_{s\le t})\mathds{1}_{H_{\frac{1}{2T}}>0}\cdot
    e^{\demi\left<\theta, W \theta\right>-\demi\left<\eta, ({K_T})^{-1} T\eta\right>+\left<\eta,\theta\right>}{1\over \sqrt{\vert K_{T}\vert}}\right)
  \\
  &=
    \E_{V,\theta}\left({M_t\over M_0}\phi((X_s)_{s\le t})\mathds{1}_{H_{\frac{1}{2T}}>0}\cdot
    e^{\demi\left<X(t), \widetilde W^{(t)} X(t)\right>-\demi\left<\widetilde \eta^{(t)}, (\widetilde H_{\widetilde\beta}^{(t)})^{-1} \widetilde \eta^{(t)}\right>+\left<\widetilde \eta^{(t)},X(t)\right>}{1\over \sqrt{\vert \widetilde K^{(t)}_{\widetilde T}\vert}}\right)
\end{align*}
Let us denote by $\left<\cdot,\cdot\right>_{V(t)}$ the usual
scalar product on $\R^{V(t)}$ (we keep denoting by
$\left<\cdot,\cdot\right>$ the usual scalar product on $\R^V$).  As
\(X(t)\) vanishes on \(V\setminus V(t)\), we have
\[\left<X(t), \widetilde W^{(t)} X(t)\right>=\left<X(t), \widetilde W^{(t)}
  X(t)\right>_{V(t)},\ \
\left<\widetilde \eta^{(t)},X(t)\right>=\left<\widetilde \eta^{(t)},X(t)\right>_{V(t)},\]
By \eqref{eq-tilhtilbeta-1}, since
$(\widetilde H_{\widetilde\beta}^{(t)})^{-1}=(\widetilde K_{\widetilde T}^{(t)})^{-1}\widetilde T^{(t)}$ and
since $\widetilde T^{(t)}$ vanishes on the subset $V\setminus V(t)$
and $\widetilde H_{\widetilde\beta}^{(t)}$ is symmetric, we get
\[\left< \widetilde{\eta}^{(t)}, (\widetilde{H}_{\widetilde{\beta}}^{(t)})^{-1}
  \widetilde{\eta}^{(t)} \right>=\left< \widetilde{\eta}^{(t)},
  (\widetilde{H}_{\widetilde{\beta}}^{(t)})^{-1} \widetilde{\eta}^{(t)}
\right>_{V(t)} ,\]
Moreover,
\[|\widetilde K^{(t)}_{\widetilde T}|=|\operatorname{Id}-\widetilde T^{(t)}
\widetilde W^{(t)}|=|(\operatorname{Id}-\widetilde T^{(t)}\widetilde W^{(t)})_{V(t),V(t)}|\]
and
\[\mathds{1}_{H_{\frac{1}{2T}}>0}=\mathds{1}_{H_{\beta^{(t)}}>0}
\mathds{1}_{\widetilde{H}_{\widetilde{\beta}}^{(t)}>0}\] thus
\[\mathds{1}_{H_{\frac{1}{2T}}>0} e^{\demi\left<X(t), \widetilde W^{(t)}
    X(t)\right>-\demi\left<\widetilde \eta^{(t)}, (\widetilde
    H_{\widetilde\beta}^{(t)})^{-1}
    \widetilde \eta^{(t)}\right>+\left<\widetilde \eta^{(t)},X(t)\right>}{1\over
  \sqrt{\vert \widetilde K^{(t)}_{\widetilde T}\vert}} =
\mathds{1}_{H_{\beta^{(t)}}>0} \frac{d
  \overline\P_{V(t)}^{\widetilde W^{(t)},X(t),\widetilde{\eta}^{(t)}}}{d
  \P_{V(t),X(t)}}.\] Therefore,
\begin{align*}
  \overline\E^{W,\theta,\eta}_{V}\left(\phi((X_s)_{s\le t})\right))
  &=
    \E_{V,\theta}\left(\mathds{1}_{H_{\beta^{(t)}}>0}{M_t\over M_0}\phi((X_s)_{s\le t})
    \overline E_{V(t)}^{\widetilde W^{(t)},{X(t)},\widetilde \eta^{(t)}}\left(1\right)\right)
  \\
  &=
    \E_{V,\theta}\left(\mathds{1}_{H_{\beta^{(t)}}>0}\phi((X_s)_{s\le t})e^{\int_0^t \left<W\psi(s)+\eta, dX_s\right>-\demi\int_0^t\left<W\psi(s)+\eta, \mathds{1}_{s<T}(W\psi(s)+\eta)\right>ds}
    \right)
\end{align*}
where we used Lemma \ref{calcul-Ito} in the second equality.  It
implies that
$$
\overline\P^{W,\theta,\eta}_{V}=\mathds{1}_{H_{\beta^{(t)}}>0}\exp\left({\int_0^t
  \left<W\psi(s)+\eta, dX_s\right>-\demi\int_0^t\left<W\psi(s)+\eta,
    \mathds{1}_{s<T}(W\psi(s)+\eta)\right>ds}\right) \P_{V,\theta}.
$$
Finally, by Girsanov's theorem, we know that under the law
\begin{align}\label{Girsanov-law}
  \exp\left({\int_0^t \left<W\psi(s)+\eta, dX_s\right>-\demi\int_0^t\left<W\psi(s)+\eta, \mathds{1}_{s<T}(W\psi(s)+\eta)\right>ds}\right) \P_{V,\theta}
\end{align}
the process
\[\left(\widetilde B(t)\right)_{t\ge 0}:= \left(X_{t}+\int_{0}^{t}
\mathds{1}_{s<T}(W\psi(s)+\eta) ds\right)_{t\ge 0}\] is a Brownian motion
stopped at time \(T\), the first hitting time of $0$ by $(X(t))$.
(Indeed, recall that $\P_{V,\theta}$ is the law of independent
Brownian motions starting at $\theta$ and stopped at their first
hitting time of $0$).  Hence,
$$
d X(t)=\mathds{1}_{t<T} d \widetilde B(t)+ \mathds{1}_{t<T}(W\psi(t)+\eta)
dt,
$$
and under the law \eqref{Girsanov-law}, $X$ is solution of the S.D.E
\ref{SDE0X} with driving Brownian motion $\widetilde B$.  By
Lemma~\ref{def-SDE0}, we know that a.s. under the law
\eqref{Girsanov-law}, we have $H_{\beta^{(t)}}>0$, thus
$\overline\P^{W,\theta,\eta}_{V}$ and \eqref{Girsanov-law} are
equal. Hence, under $\overline\P^{W,\theta,\eta}_{V}$, $(X(t))$ has
the law of the solutions of the S.D.E \ref{SDE0X}.

\end{proof}

\section{Proof of the Abelian properties : Theorem~\ref{thm_abelian}}
\label{sec_abelian}
\begin{proof}[Proof of
Theorem~\ref{thm_abelian}~(\ref{thm_abelian-i}),~(\ref{thm_abelian-ii})]
Consider first the restriction property~(\ref{thm_abelian-i}).  By
Theorem~\ref{thm_main1}, conditionally on \(T\),
\((X_{i}(t))_{i\in V}\) are independent Bessel bridges from $\theta_i$
to 0 in time $T_i$.  By Theorem~\ref{thm_main1} and
Lemma~\ref{lem_marg_cond}, \(1\over 2 T_{U}\) is
\(\nu_{U}^{W_{U,U},\theta_{U},\widecheck{\eta}}\) distributed.
By Theorem~\ref{thm_main1} applied to the set $U$ and parameters
$W_{U,U}$, $\theta_{U}$, $\widecheck{\eta}$, it implies that $X_U$ has the
law of the solutions of
\hyperref[SDE0X]{$E^{W,\theta,\widecheck\eta}_U(X)$}.

For (\ref{thm_abelian-ii}), the same argument applies, using that
\(\beta_{U^{c}}\), conditionally on \(\beta_{U}\), is
\(\nu_{U^{c}}^{\widecheck W,\theta_{U^{c}}, \widecheck\eta}\) distributed.
\end{proof}

\begin{proof}[Proof of Theorem~\ref{thm_abelian}~(\ref{thm_abelian-iii})]
Recall that we denote by $X(t)$ (resp. $(X_i(t))_{i\in V}$) the
canonical process on Wiener space $\bold W=C(\R_+,\R)$
(resp. $\bold{W}_V=C(\R_+, \R^V)$).  Recall that
\(\mathbb{B}_{\theta,0}^{3,T}\) and \(\mathbb{E}_{\theta,0}^{3,T}\)
denotes the law (resp. the expectation) on $\bold W$ of a Bessel
bridge from $\theta$ to 0 on time interval $[0,T]$ (and extended by 0
for $t\ge T$). Recall also that $\E_{V}^{W,\theta,\eta}(\cdot)$
denotes the expectation with respect to the law on $\bold{W}_V$ of the
solution of the S.D.E. \ref{SDE0X}.

Following \cite{Revuz-Yor}~p.463, under
\(\mathbb{B}_{\theta,0}^{3,T}\), the law of \(X(t)\) for some
\(0< t< T\) is given by $p_{\theta,0}^{3,t,T}(y)dy$ on $\R_+$, with
\begin{align}
\label{Bessel_kernel}
                                   p_{\theta,0}^{3,t,T}(y)=\frac{1}{\sqrt{2\pi t}}\frac{y}{\theta}\left( \frac{T}{T-t} \right)^{3/2}e^{-\frac{y^{2}}{2(T-t)}+\frac{\theta^{2}}{2T}}\left( e^{-\frac{(y-\theta)^{2}}{2t}}-e^{-\frac{(y+\theta)^{2}}{2t}} \right), \;\;\; \forall y\ge 0.
\end{align}
Moreover, the Markov property of the Bessel bridge implies that under
\(\mathbb{B}_{\theta,0}^{3,T}\) and conditionally on $X(t)=x$,
$0<t<T$, the law of $((X(u))_{0\le u\le t}, (X(t+u))_{0\le u\le T-t})$
is given by
\begin{align}
  \label{Markov}
  \mathbb{B}_{\theta,x}^{3,t}\otimes  \mathbb{B}_{x,0}^{3,T-t}.
\end{align}

Let us denote by $\overline \nu_V^{W,\theta,\eta}(dT)$ the law of
$T={1\over 2\beta}$ when $\beta$ follows the law
$\nu_V^{W,\theta,\eta}(d\beta)$, so that
\[\overline\nu_{V}^{W,\theta,\eta}(dT)=\mathds{1}_{H_{\frac{1}{2T}}>0}\left(
  \frac{2}{\pi} \right)^{|V|/2}e^{-\frac{1}{2}\left<
    \theta,\frac{1}{T}\theta \right>+\frac{1}{2}\left< \theta,W\theta
  \right>-\frac{1}{2}\left< \eta,(H_{\frac{1}{2T}})^{-1}\eta
  \right>+\left< \eta,\theta \right>}\frac{\prod_{i\in
    V}\theta_{i}}{\sqrt{|H_{\frac{1}{2T}}|}}\prod_{i\in V}\frac{1}{2
  T_{i}^{2}}dT_i\]

Let $(t_i^0)_{i\in V} \in \R_+^V$ be as in the statement of the
theorem. Set
\begin{eqnarray}\label{Vt}
V(t^0)=\{i\in V, \; T_i>t_i^0\}.
\end{eqnarray}
Fix $U\subset V$, and 
denote by $\aaa(t^0,T)$ the
event
\begin{eqnarray}\label{At}
\aaa(t^0,T)=\{V(t^0)=U\}=\{T> t_{i}^{0},i\in U\}\cap
\{ T\le t_{i}^{0},i\in U^c\}
\end{eqnarray}
Let \(h,g\) be bounded measurable test functions. 
By Theorem~\ref{thm_main1}, we have
\begin{align*}
  &  \E_{V}^{W,\theta,\eta}\left[\mathds{1}_{V(t^0)=U} h((X_{i}[0,t_i^0])_{i\in V}) g((X_{i}([t_{i}^{0},T_{i}]))_{i\in U}) \right]\\
  &=\int \mathds{1}_{\aaa(t^0,T)}  \bigotimes_{i\in {V}}\E_{\theta_{i},0}^{3,T_{i}}\left[h((X_{i}[0,t_i^0])_{i\in V}) g((X_{i}([t_{i}^{0},T_{i}]))_{i\in U})  \right]
    d\overline\nu_{V}^{W,\theta,\eta}(T)
\end{align*}
By the Markov property \eqref{Markov}, we have on the event
$\aaa(t^0,T)$, that
\begin{eqnarray*}
  && \bigotimes_{i\in {V}}\E_{\theta_{i},0}^{3,T_{i}}\left[h((X_{i}[0,t_i^0])_{i\in V}) g((X_{i}([t_{i}^{0},T_{i}]))_{i\in U})  \right]
  \\
  &=& \int_{\R_+^{U}} K(x_{U}, t^0_{U}, T_{U^c})
      \bigotimes_{i\in {U}}\E_{x_i,0}^{3,T_{i}-t^0_i}\left[g((X_{i}([0,T_{i}-t_{i}^{0}]))_{i\in U})\right]
      \left(\prod_{i\in U} p_{\theta_i,0}^{3,t_i^0,T_i}(x_i) dx_i\right)
\end{eqnarray*}
where
\begin{align}
\label{eq-H}
 K(x_{U}, t^0_{U}, T_{U^c})=
 \left(\bigotimes_{i\in  {U}}\E_{\theta_{i},x_i}^{3,t^0_{i}}
   \bigotimes_{i\in U^c}\E_{\theta_{i},0}^{3,T_{i}}\right)\left[
   h((X_{i}[0,t_i^0])_{i\in V})\right]
\end{align}
is a function that only depends on
\((x_i, t^0_i)_{i\in U}, (T_i)_{i\in U^c}\).
We thus get,
\begin{align*}
  &  \E_{V}^{W,\theta,\eta}\left[\mathds{1}_{V(t^0)=U} h((X_{i}[0,t_i^0])_{i\in V}) g((X_{i}([t_{i}^{0},T_{i}]))_{i\in U}) \right]\\
  &=
 \int \mathds{1}_{\aaa(t^0,T)} K(x_{U}, t^0_{U}, T_{U^c})
      \bigotimes_{i\in {U}}\E_{x_i,0}^{3,T_{i}-t^0_i}\left[g((X_{i}([0,T_{i}-t_{i}^{0}]))_{i\in U})\right]
      \left(\prod_{i\in U} p_{\theta_i,0}^{3,t_i^0,T_i}(x_i) dx_i\right)   d\overline\nu_{V}^{W,\theta,\eta}(T)
\end{align*}
In the sequel, on the event $\aaa(t^0,T)$, we set
$$
(\widetilde T_i)_{i\in U}=(T_i-t_i^0)_{i\in U}.
$$
The strategy is now to show that we can combine the terms $\prod_{i\in U} p_{\theta_i,0}^{3,t_i^0,T_i}(x_i)$ and the measure $d\overline\nu_{V}^{W,\theta,\eta}(T)$
in such a way that on the event $\aaa(t^0,T)$, changing from variables $(T_i)_{i\in U}$ to variables $(\widetilde T_i)_{i\in U}$, we end up with a function of 
$(x_{U}, t^0_{U}, T_{U^c})$ and the measure $\nu_{U}^{\widetilde{W}^{(t^0)},x,\widetilde{\eta}^{(t^0)}}(d \widetilde{T})$, see forthcoming formula~\eqref{T->tilde_T}.

Let us denote by $\left< \cdot,\cdot \right>_{U}$ the usual scalar product
on $\R^{U}$ (recall that we keep denoting by $\left< \cdot,\cdot \right>$ the usual
scalar product on $\R^V$).  Note that $\widetilde\eta^{(t^0)}$ and $\widetilde W^{(t^0)}$
defined in Theorem~\ref{thm_abelian}~(\ref{thm_abelian-iii}) correspond to
$\widetilde \eta$ and $\widetilde W$ of Lemma~\ref{Kt+s} for $ t^0\wedge T$
and $\widetilde T$.  Hence, by \eqref{eta-H} of Lemma~\ref{Kt+s}, we get, with $\widetilde H_{\frac{1}{2 \widetilde{T}}}^{(t^0)}=\frac{1}{2 \widetilde{T}}-\widetilde W^{(t^0)}$,
that
$$
\left< \widetilde{\eta}^{(t^0)},(\widetilde{H}_{\frac{1}{2 \widetilde{T}}}^{(t^0)})^{-1}
  \widetilde{\eta}^{(t^0)} \right>_{U}-\left<
  \eta,(H_{\frac{1}{2T}})^{-1}\eta \right>=-\left<
  \eta,(H_{\frac{1}{2t^{0}\wedge T}})^{-1}\eta \right>
$$
and by \eqref{eq-det} of Lemma~\ref{Kt+s}
$$
\frac{\left| \left(\widetilde{H}_{\frac{1}{2 \widetilde{T}}}^{(t^0)}\right)_{U,U}\right|}{| H_{\frac{1}{2T}}|}=\left| K_{t^0\wedge
    T}\right| \prod_{i\in U}\left( \frac{T_{i}}{\widetilde{T}_{i}}
\right).
$$
Note that we have
\begin{align*}
                                   \prod_{i\in U} p_{\theta_{i},0}^{3,t_i^0,T_{i}}(x_{i})=e^{-\demi\left<x,{1\over \widetilde T}x\right>_{U}+\demi \left<\theta,\frac{1}{T}\theta\right>_{U}}  \prod_{i\in U}\left( e^{-\frac{(x_i-\theta_i)^{2}}{2t_i^0}}-e^{-\frac{(x_i+\theta_i)^{2}}{2t^0_i}} \right)\frac{1}{\sqrt{2\pi t_i^0}}\frac{x_i}{\theta_i}
                                   \left( \frac{T_{i}}{\widetilde{T}_{i}} \right)^{3/ 2}\\
\end{align*}
Changing from variables $(T_i)_{i\in U}$ to
$(\widetilde T_i)_{i\in U}$, we get
\begin{align}
\label{T->tilde_T}
\mathds{1}_{\aaa(t^0,T)}\left(\prod_{i\in U}
  p_{\theta_{i},0}^{3,T_{i}}(x_{i})\right)
\overline\nu_{V}^{W,\theta,\eta}(dT) =\mathds{1}_{T_i<t_i^0, \; i\in
  U^c} \Xi(x_{U}, t^0_{U}, T_{U^c})
\nu_{U}^{\widetilde{W}^{(t^0)},x,\widetilde{\eta}^{(t^0)}}(d \widetilde{T}) \prod_{i\in
  U^{c}} dT_{i} 
  \end{align} 
for some explicit function
\(\Xi(x_{U}, t^0_{U}, T_{U^c})\) that only 
depends on \((x_i, t^0_i)_{i\in U}, (T_i)_{i\in U^c}\).

Continuing our computation, we have 
\begin{align}
\nonumber    & \E_{V}^{W,\theta,\eta}\left[\mathds{1}_{V(t^0)=U} h((X_{i}[0,t_i^0])_{i\in V}) g((X_{i}([t_{i}^{0},T_{i}]))_{i\in U}) \right]\\
   \label{equation-affreuse}
                  &= \int \mathds{1}_{T_i<t_i^0, \; i\in U^c} K(x_{U}, t^0_{U}, T_{U^c}) \Xi(x_{U}, t^0_{U}, T_{U^c})
                    \E_{U}^{\widetilde{W}^{(t^0)},x,\widetilde{\eta}^{(t^0)}}\left[ g((X_i([0,T_i]))_{i\in U})  \right] \prod_{i\in U} dx_i \prod_{i\in U^c}dT_{i}
\end{align}

Let us apply the last equality to the case where $h$ and $g$ are replaced by
\begin{align*}
&\widetilde h((X_{i}[0,t_i^0])_{i\in V}):= h((X_{i}[0,t_i^0])_{i\in V})
\mathbb{E}_{U}^{\widetilde{W}^{(t^0)},X_{U}(t^{0}),\widetilde{\eta}^{(t^0)}}
\left(g((X_{i}([{0},T_{i}]))_{i\in
  U})
  \right),
  \\
  &\widetilde g:=1
\end{align*}
The identity \eqref{equation-affreuse} gives in this case
\begin{align}
\nonumber
& \E_{V}^{W,\theta,\eta}\left[\mathds{1}_{V(t^0)=U} h((X_{i}[0,t_i^0])_{i\in V}) \mathbb{E}_{U}^{\widetilde{W}^{(t^0)},X_{U}(t^{0}),\widetilde{\eta}^{(t^0)}}
\left( g((X_{i}([{0},T_{i}]))_{i\in U})\right) 
   \right]
 \\
\label{affreuse-2}  & =\int \mathds{1}_{T_i<t_i^0, \; i\in U^c} \widetilde K(x_{U}, t^0_{U}, T_{U^c}) \Xi(x_{U}, t^0_{U}, T_{U^c})
     \prod_{i\in U} dx_i \prod_{i\in U^c}dT_{i}
\end{align}
where, using \eqref{eq-H} applied to $\widetilde h$ instead,
$$
\widetilde K(x_{U}, t^0_{U}, T_{U^c})= K(x_{U}, t^0_{U}, T_{U^c})\E_{U}^{\widetilde{W}^{(t^0)},x,\widetilde{\eta}^{(t^0)}}\left[ g((X([0,T])_{i\in U}))  \right].
$$
Remark that the right-hand sides of \eqref{equation-affreuse} et \eqref{affreuse-2} are thus the same.
Hence, we conclude that
\begin{align*}    & \E_{V}^{W,\theta,\eta}\left[\mathds{1}_{V(t^0)=U} h((X_{i}[0,t_i^0])_{i\in V}) g((X_{i}([t_{i}^{0},T_{i}]))_{i\in U}) \right]\\
                  &=
                    \E_{V}^{W,\theta,\eta}\left[\mathds{1}_{V(t^0)=U} h((X_{i}[0,t_i^0])_{i\in V})
                    \mathbb{E}_{U}^{\widetilde{W}^{(t^0)},X_{U}(t^{0}),
                    \widetilde{\eta}^{(t^0)}}\left(g((X_{i}([0,T_{i}]))_{i\in
                    U})\right) \right].
\end{align*}
Summing on all possible choices of $U$, we exactly get that the
law of $(X_{i}([t_{i}^{0},T_{i}]))$, conditionally on $\fff^X({t^0})$,
is the law of the solutions of the
S.D.E. \hyperref[SDE0X]{\(E_{V}^{\widetilde W^{(t^0)},X(t^0),\widetilde \eta^{(t^0)}}(X)\)}.

\end{proof}

\begin{proof}[Proof of Theorem~\ref{thm_abelian}~(\ref{thm_abelian-iiii})]
Fix as before $U\subset V$. With the notations \eqref{Vt} and \eqref{At}, we have
\[
V(T^0)=\{i\in V, \ T_i>T^0_i\}, \;\;\; \aaa(T^0,T)=\{V(T^0)=U\}.
\]
We simply write $\{T_0<\infty\}$ for the event $\{T^0_i<\infty, \forall i\in V\}$.
In order to prove the strong Markov property~\eqref{thm_abelian-iiii}, it is enough to prove that, for any bounded test function \(h,g\), depending continuously on finitely many marginals of $X$, we have
\begin{equation}
\label{eq-markov-peroperty-intermoftestfunc}
\begin{aligned}
&\mathbb{E}_V^{W,\theta,\eta}\left[ \mathds{1}_{T^0<\infty}\mathds{1}_{\aaa(T^0,T)} h((X_i[0,T^0_i])_{i\in V})g((X_i[T^0_i,T_i])_{i\in U}) \right]\\
&=\mathbb{E}_V^{W,\theta,\eta}\left[ \mathds{1}_{T^0<\infty}\mathds{1}_{\aaa(T^0,T)}  h((X_i[0,T^0_i])_{i\in V})\mathbb{E}_{U}^{\widetilde{W}^{(T_0)},X_{U}(T^0),\widetilde{\eta}^{(T^0)}}(g((X_i[0,T_i])_{i\in U})) \right]
\end{aligned}
\end{equation}
We define the sequence of stopping times, for all $i\in V$, by
\[[T^0_i]_n= \frac{k}{2^n} \text{ when }\frac{k-1}{2^n}\le T^0_i<\frac{k}{2^n},\ k\in \mathbb{N},\]
and $[T^0_i]_n=\infty$ when $T^0_i=\infty$.
We can check that \([T^0]:=([T^0_i]_n)_{i\in V} \) is a multi-stopping time in the sense of Theorem~\ref{thm_abelian}~(\ref{thm_abelian-iiii}), since for $(k_i)_{i\in V}\in \N^V$,
 \[
 \bigcap_{i\in V} \{\frac{k_i-1}{2^n}\le T^0_i<\frac{k_i}{2^n}\}\in \sigma\left(X_i(s),\ s\le \frac{k_i}{2^n},\; i\in V\right).
 \] 
Moreover, \([T^0_i]_n\) decreases a.s. to \(T^0_i\) and for $n$ large enough $V([T^0]_n)=V(T^0)$ a.s.. This implies that a.s.
\[
\mathds{1}_{T^0<\infty}
\mathds{1}_{\aaa(T^0,T)} g((X_i[T^0_i, T_i])_{i\in U})=\lim_{n\to \infty} 
\mathds{1}_{[T^0]_n<\infty}
\mathds{1}_{\aaa([T^0]_n,T)}g((X_i[[T^0_i]_n,T_i])_{i\in U}).\]
Therefore, by dominated convergence theorem,
\begin{align*}
  &\mathbb{E}_V^{W,\theta,\eta}\left[ \mathds{1}_{T^0<\infty}\mathds{1}_{\aaa(T^0,T)} h((X_i[0,T^0_i])_{i\in V})g((X_i[T^0_i,T_i])_{i\in U}) \right]\\
  &=\lim_{n\to \infty}\mathbb{E}_V^{W,\theta,\eta}\left[ \mathds{1}_{[T^0]_n<\infty}
\mathds{1}_{\aaa([T^0]_n,T)} h((X_i[0,T^0_i])_{i\in V}) g((X_i[[T^0_i]_n,T_i])_{i\in U}) \right]\\
  &=\lim_{n\to \infty}\sum_{k=(k_i)_{i\in V}\in \mathbb{N}^V}\mathbb{E}_V^{W,\theta,\eta}\left[ \left(\prod_{i\in V}\mathds{1}_{\frac{k_i-1}{2^n}\le T^0_i<\frac{k_i}{2^n}}\right) \mathds{1}_{\aaa({k\over 2^n},T)} h((X_i[0,T^0_i])_{i\in V})g((X_i[\frac{k_i}{2^n},T_i])_{i\in U}) \right]
\end{align*}
where in the last equality we sum on the possible values of each \([T^0_i]_n,i\in V\). Note that
\[\left(\prod_{i\in V}\mathds{1}_{\frac{k_i-1}{2^n}\le T^0_i<\frac{k_i}{2^n}}\right) \mathds{1}_{\aaa({k\over 2^n},T)}  h((X_i[0,T^0_i])_{i\in V})\]
is \(\mathcal{F}^X(\frac{k}{2^n})\) measurable, so we can apply the Markov property~\eqref{thm_abelian-iii}, and we get
\begin{align*}
&  \mathbb{E}_V^{W,\theta,\eta}\left[ \mathds{1}_{\aaa({k\over 2^n} ,T)} \left(\prod_{i\in V}\mathds{1}_{\frac{k_i-1}{2^n}\le T^0_i<\frac{k_i}{2^n}}\right) h((X_i[0,T^0_i])_{i\in V})g((X_i[\frac{k_i}{2^n},T_i])_{i\in V(\frac{k}{2^n})}) \right]\\
  &=\mathbb{E}_V^{W,\theta,\eta}\left[ \mathds{1}_{\aaa({k\over 2^n},T)} \left(\prod_{i\in V}\mathds{1}_{\frac{k_i-1}{2^n}\le T^0_i<\frac{k_i}{2^n}}\right) h((X_i[0,T^0_i])_{i\in V}) \mathbb{E}_{U}^{\widetilde{W}^{(\frac{k}{2^n})},X_{U} (\frac{k}{2^n}),\widetilde{\eta}^{(\frac{k}{2^n})}}(g((X_i[0,T_i])_{i\in U})) \right].
\end{align*}
Summing on possible values of $(k_i)$, we get:
\begin{align}
\nonumber
  &\mathbb{E}_V^{W,\theta,\eta}\left[ \mathds{1}_{T^0<\infty}\mathds{1}_{\aaa(T^0,T)} h((X_i[0,T^0_i])_{i\in V})g((X_i[T^0_i,T_i])_{i\in U}) \right]\\
 \label{ref_g}
  &=\lim_{n\to \infty} \mathbb{E}_V^{W,\theta,\eta}
  \left[ \mathds{1}_{T^0<\infty} \mathds{1}_{\aaa([T^0]_n,T)}  h((X_i[0,T^0_i])_{i\in V}) 
  \mathbb{E}_{U}^{\widetilde{W}^{([T^0]_n)},X_{U} ([T^0]_n),\widetilde{\eta}^{([T^0]_n)}}(g((X_i[0,T_i])_{i\in U})) \right].
\end{align}
We conclude the proof thanks to the Feller property
(see e.g. Section~18.6 of \cite{schilling2014brownian}) proved in the Lemma below.
\end{proof}
\begin{lem}\label{Feller}
The function $(W,\theta,\eta)\to \mathbb{E}_{V}^{W,\theta,\eta}(g((X_i[0,T_i])_{i\in V}))$ is continuous on $(\R_+^*)^E\times (\R_+^*)^V\times \R_+^V$ for any bounded measurable function $g$ depending only on a finite number of marginals.
\end{lem}
\begin{proof}[Proof of Lemma~\ref{Feller}]
It is enough to consider the case \(\eta=0\), since the case \(\eta\ne 0\) is a marginal of the case \(\eta=0\) by Lemma~\ref{lem_marg_cond}. Without loss of generality we assume \(W_{i,i}=0,\ \forall i\). The proof follows from the representation Theorem~\ref{thm_main1} and the two ingredients below.

Under the 3-dimensional Bessel bridge law, the expectation $\mathbb{E}_{\theta,0}^{3,T}(g((X_i[0,T_i])_{i\in V}))$ is continuous in $(\theta, T)$. Indeed, the 3-dimensional Bessel bridge is the norm of a 3-dimensional Brownian bridge from $x$ to $0$ if $\| x\|=\theta$, and the 3-dimensional Brownian bridge from $x$ to 0 can be represented as $x+B_t^{(3)}-{t\over T}B^{(3)}_T-{t\over T}x$ where $(B^{(3)}_t)$ is a 3-dimensional standard Brownian motion.

On the other hand, the measure $\nu_V^W(d\beta)$ can be dominated locally on the parameters $W,\theta$ after some change of coordinates, following \cite{STZ15}. (Note that the density \(\nu_V^{W,\theta}\) in the present paper correspond to \(\nu^{W,\theta^2}\) in \cite{STZ15}.) 
For convenience, write  \(V=\{1 ,\ldots,N \}\). By the change of variables $(\beta_i)_{i\in V} \to (x_i)_{i\in V}$ from $\{\beta, \; H_\beta>0\}$ to $(\R_+^*)^V$ described in the proof of Theorem 1 of \cite{STZ15} (see page 3977), we have
\begin{equation}
\label{eq-change-of-va-to-x}
\begin{aligned}
&\mathds{1}_{H_\beta>0} \exp\left( -\frac{1}{2}\left< \theta,H_{\beta}\theta \right> -\frac{1}{2}\sum_{i,j}W_{i,j}\theta_i \theta_j\right)\frac{1}{\sqrt{\det H_{\beta}}}d \beta\\
&=\frac{1}{2^N}\mathds{1}_{x\in \mathbb{R}_+^{N}}\exp\left( -\sum_{l=1}^N \left( \frac{\theta_l^2 x_l}{2}+\frac{1}{2x_l}\left( \sum_{k=l+1}^N \theta_k^2 H_{l,k}^2 \right) \right) \right)\frac{1}{\sqrt{x_1 \cdots x_N}}dx.
\end{aligned}
\end{equation}
following the notation there, in particular the definition of \(\{x_i,H_{i,j}:\ 1\le i,j\le N\}\). By definition, for any \(l\ge 1\), \(H_{l,k}\ge W_{l,k}\).

Now fix \(W^0,\theta^0\), let \(\Omega\) be a neighborhood of \((W,\theta)\), denote
\[\underline{W}_{l,k}=\inf_{\Omega}W_{l,k},\ \ \ \underline{\theta}_l=\inf_{\Omega}\theta_l.\]
For any \(W,\theta\in \Omega\), we have \(H_{l,k}\ge W_{l,k}\ge \underline{W}_{l,k}\) and \(\theta_l\ge \underline{\theta}_l \) for all \(1\le l,k\le N\), so the density in (\ref{eq-change-of-va-to-x}) is locally uniformly bounded (in the variables \(x\)s) by
\[\mathds{1}_{x\ge 0}\exp\left( -\sum_{l=1}^N \left( \frac{\underline{\theta}_l^2 x_l}{2}+\frac{1}{2x_l}\left( \sum_{k=l+1}^N \underline{\theta}_k^2 \underline{W}_{l,k}^2 \right) \right) \right)\frac{1}{\sqrt{x_1 \cdots x_N}},\]
which is an integrable function, as \(x_1 ,\ldots,x_{N-1}\) are distributed as inverse of IG distribution, and \(x_N\) is a Gamma distributed random variable.


\end{proof}

\section{Relation with the martingales associated with the
  VRJP}\label{sec_martingale}
Consider in this section that \(V\) is infinite and that \(W\) is such
that the associated graph \(\mathcal{G}\) has finite degree at each vertex
and is connected.  Following \cite{ST15}, we extend the definition of
the distribution $\nu_V^{W,\theta}$ to the case of this infinite
graph.  We assume to be coherent with \cite{ST15} that \(W\) is zero
on the diagonal.  Note that we slightly generalize the definition of
\cite{ST15} since we consider a general vector
\((\theta_i)_{i\in V}\in (\R_+)^V\), which is equal to \(1\) in
\cite{ST15}. (But as noted at the beginning of section
\ref{sec_results_distrib} it is in fact not more general since we can
always take $\theta$ to $1$ by a change of variables on $\beta$ and
$W$.)

Let us recall the construction of the distribution $\nu_V^{W,\theta}$
obtained by Kolmogorov's extension Theorem. The approach is slightly
different from that of \cite{ST15} and make use of
Lemma~\ref{lem_marg_cond}, \eqref{item-2}.  Let \(V_n\) be an
increasing sequence of subsets such that \(\cup_{n\ge 1} V_n= V\).
Consider the vector \(\eta^{(n)}\in (\R_+)^{V_n}\) defined by
\begin{align}\label{etan}
\eta^{(n)}=W_{V_n,V_n^c}(\theta_{V^c_n}).
\end{align}
By Lemma~\ref{lem_marg_cond}, \eqref{item-2}, the sequence of
distribution \(\nu_{V_n}^{W,\eta^{(n)}}\) is compatible, hence by
Kolmogorov theorem it can be extended to a measure
\(\nu^{W,\theta}_{V}\) on \((\R_+)^V\).  We define the Schr\"odinger
operator 
\[
H_\beta:=2\beta-W,
\]
on \(\R^V\) associated with the potential
\(\beta\sim\nu^{W,\theta}_V\). Note that $H_\beta\ge 0$ as the limit
of $(H_{\beta})_{V_n,V_n}$ which is positive definite since \(\beta_{V_n}\) has law \(\nu_{V_n}^{W,\theta,\eta^{(n)}}\).

In \cite{SZ15} we considered the sequence of functions
\((\psi^{(n)}_j)_{j\in V}\in (\R_+)^V\) defined by
\begin{align}\label{psi}
  \begin{cases}
  (H_\beta \psi^{(n)})_{V_n}=0
  \\
  \psi^{(n)}_{V_n^c}=\theta_{V_n^c}
  \end{cases}
\end{align}
and the operators \((\widehat G^{(n)}(i,j))_{i,j\in V_n}\) by
\[
\begin{cases}
\widehat G^{(n)}_{V_n,V_n}=((H_\beta)_{V_n,V_n})^{-1},
\\
\widehat G^{(n)}(i,j)=0, \hbox{ if \(i\) or \(j\) in not in \(V_n\)}
\end{cases}
\]
Let \(\fff_n=\sigma(\beta_i, \;i\in V_n)\), the sigma field generated
by \(\beta_{V_n}\).  In \cite{SZ15}, Proposition~9, it was proved that
\(\psi^{(n)}\) is a vectorial \(\fff_n\)-martingale, with quadratic variation
given by \(\widehat G^{(n)}(i,j)\), i.e. that for all \(i,j\) in \(V\) and
all \(n\)
\[
\E\left( \psi^{(n+1)}(i)\psi^{(n+1)}(j)-\widehat G^{(n+1)}(i,j) |
  \fff_n\right)= \psi^{(n)}(i)\psi^{(n)}(j)-\widehat G^{(n)}(i,j).
\]
It was extended in \cite{DMR15} to an exponential martingale property,
namely it was proved that for any compactly supported function
\(\lambda\in (\R_+)^V\),
\begin{align}\label{exp-mart}
e^{-\left<\lambda,\psi^{(n)}\right>-\demi\left<\lambda, \widehat G^{(n)}
    \lambda\right>},
\end{align}
is a \(\fff_n\)-martingale. 


We can interpret the functions $\psi^{(n)}$ that appear above in terms of the S.D.E.s.
Consider \( X^{(n)}\)
the solution of the S.D.E. \hyperref[SDE0X]{$E^{W,\theta,\eta^{(n)}}_{V_n}$}, where $\eta^{(n)}$ is defined in \eqref{etan}. Denote by 
$T^{(n)}$ the associated stopping times and $\beta^{(n)}={1\over 2 T^{(n)}}$ and 
$$
K^{(n)}_{t\wedge T^{(n)}}=\operatorname{Id}_{V_n,V_n}-(t\wedge T^{(n)}) W_{V_n,V_n}, \;\;\; \psi^{(n)}(t)=\left(K^{(n)}_{t\wedge T^{(n)}}\right)^{-1}X^{(n)}(t),
$$
the associated operator and martingale that appear in Lemma~\ref{def-SDE0}. We always consider that $\psi^{(n)}$ is extended to the full set $V$ by $\psi^{(n)}_{V_n^c}(t)=\theta_{V_n^c}$. Considering \eqref{psi}, we have that
$$
\lim_{t\to\infty} \psi^{(n)}(t)=\psi^{(n)}.
$$
Hence the function $\psi^{(n)}$ appears as the limit of the continuous martingale $\psi^{(n)}(t)$.
 
It is possible to interpret the exponential martingale property \eqref{exp-mart} in terms of the Abelian properties, see Theorem~\ref{thm_abelian}.   
More precisely, conditionally on $\sigma(\beta_{V_n})$, it is possible to construct a continuous martingale that interpolates between
$\psi^{(n)}$ and $\psi^{(n+1)}$ and with total quadratic variation given by $\widehat G^{(n+1)}-\widehat G^{(n)}$, which explains the exponential martingale property as 
a consequence the standard exponential martingale property for continuous martingales. We do not give details of this computation which requires heavy notations (but the authors will provide details under request).

\bibliography{bibi-Bessel} \bibliographystyle{plain}

\end{document}